\documentclass[12pt]{amsart}
\usepackage{amsmath}
\usepackage{amscd}
\usepackage{amssymb}
\usepackage{amsfonts}
\usepackage{amsthm}
\usepackage{bbm}
\usepackage{cancel}
\usepackage{color}
\usepackage{eucal}
\usepackage{enumerate,yfonts}
\usepackage{enumitem}
\usepackage{relsize}

 \allowdisplaybreaks

\usepackage{pdfsync}
 \usepackage[all,cmtip]{xy}
\usepackage{graphicx}
\usepackage{graphics}
\usepackage{hyperref}
\usepackage{latexsym}
\usepackage{mathrsfs}
 \usepackage{placeins}
\usepackage{pstricks}
\usepackage{bookmark}

\usepackage{stmaryrd}
\usepackage{url}

\DeclareFontFamily{U}{mathx}{\hyphenchar\font45}
\DeclareFontShape{U}{mathx}{m}{n}{
      <5> <6> <7> <8> <9> <10>
     <10.95> <12> <14.4> <17.28> <20.74> <24.88>
    mathx10
      }{}
\DeclareSymbolFont{mathx}{U}{mathx}{m}{n}
\DeclareMathAccent{\widecheck}{\mathalpha}{mathx}{"71}

\newtheorem{thm}{Theorem}[section]
\newtheorem{corollary}[thm]{Corollary}
\newtheorem{lemma}[thm]{Lemma}

\newtheorem{proposition}[thm]{Proposition}
\newtheorem{prop}[thm]{Proposition}

\newtheorem{thm-dfn}[thm]{Theorem-Definition}

\newtheorem{definition}[thm]{Definition}

\newtheorem{theorem}{Theorem}[section]

\numberwithin{equation}{section}

\newcommand{\fp}{{\mathfrak p}}

\newcommand{\fa}{{\mathfrak a}}

\newcommand{\Lp}{{\mathfrak{p}}}

\newcommand{\Lg}{{\mathfrak g}}

\newcommand{\Ln}{{\mathfrak{n}}}

\newcommand{\Ll}{{\mathfrak{l}}}

\newcommand{\nmod}{\hspace{-.1in}\mod}

\newcommand{\rf}{{\mathrm f}}
\newcommand{\rn}{{\mathrm n}}

\newcommand{\bC}{{\mathbb C}}

\newcommand{\bG}{{\mathbb G}}
\newcommand{\bZ}{{\mathbb Z}}

\newcommand{\bN}{{\mathbb N}}

\newcommand{\calF}{{\mathcal F}}

\newcommand{\calT}{{\mathcal T}}

\newcommand{\cO}{{\mathcal O}}
\newcommand{\cA}{{\mathcal A}}
\newcommand{\cF}{{\mathcal F}}
\newcommand{\cN}{{\mathcal N}}
\newcommand{\cH}{{\mathcal H}}

\newcommand{\cT}{{\mathcal T}}

\newcommand{\cP}{{\mathcal P}}
\newcommand{\cE}{{\mathcal E}}

\newcommand{\cM}{{\mathcal{M}}}

\newcommand{\cL}{{\mathcal{L}}}

\newcommand{\on}{\operatorname}

\newcommand{\Loc}{\on{LocSys}}

\newcommand{\nc}{\newcommand}

\nc{\al}{{\alpha}} \nc{\be}{{\beta}} \nc{\ga}{{\gamma}}
\nc{\ve}{{\varepsilon}} \nc{\Ga}{{\Gamma}} 
\nc{\La}{{\fa}}

\nc{\ad }{{\on{ad }}}

\nc{\aff}{{\on{aff}}} \nc{\Aff}{{\mathbf{Aff}}}

\nc{\der}{{\on{der}}}

\nc{\diag}{{\on{diag}}}

\nc{\Fl}{{\calF\ell}}

\nc{\Hg}{{\on{Higgs}}}

\nc{\Id}{{\on{Id}}}

\nc{\Ind}{{\on{Ind}}}

\nc{\Op}{{\on{Op}}}

\nc{\res}{{\on{res}}}

\nc{\tr}{{\on{tr}}}

\nc{\GSp}{{\on{GSp}}} \nc{\GU}{{\on{GU}}} \nc{\SL}{{\on{SL}}}
\nc{\SU}{{\on{SU}}} \nc{\SO}{{\on{SO}}}

\nc{\nh}{{\Loc_{J^p}(\tau')}}
\nc{\bnh}{{\Loc_{\breve J^p}(\tau')}}

\nc{\bU}{{\overline{U}}} 
\nc{\IC}{{\on{IC}}}

\nc{\op}{{\operatorname{P}}}

\newcommand{\br}{\begin{rouge}}
\newcommand{\er}{\end{rouge}}

\newcommand{\bb}{\begin{bluet}}
\newcommand{\eb}{\end{bluet}}

\newcommand{\p}{\perp}

\nc{\ot}{\otimes}

\nc{\oh}{{\operatorname{H}}}
\nc{\gr}{{\operatorname{gr}}}
\nc{\rk}{{\operatorname{rank}}}
\nc{\codim}{{\operatorname{codim}}}
\nc{\img}{{\operatorname{Im}}}
\nc{\Span}{{\operatorname{Span}}}
\nc{\Img}{\operatorname{Im}}
\nc{\Char}{\operatorname{Char}}

\newcommand{\beqn}{\begin{equation*}}
\newcommand{\eeqn}{\end{equation*}}

\newcommand{\beq}{\begin{equation}}
\newcommand{\eeq}{\end{equation}}

\newcommand{\bern}{\begin{eqnarray*}}
\newcommand{\eern}{\end{eqnarray*}}

\newcommand{\ber}{\begin{eqnarray}}
\newcommand{\eer}{\end{eqnarray}}

\newcommand{\dsum}{{\mathlarger{\subset}\mkern-12mu\raisebox{2pt}{\mbox{\larger[-5]$\mathsmaller{\bigoplus}$}}}}

\begin{document}
\title[Character sheaves for symmetric pairs: spin groups]{Character sheaves for symmetric pairs: spin groups}

         \author{Ting Xue}
         \address{ School of Mathematics and Statistics, University of Melbourne, VIC 3010, Australia, and Department of Mathematics and Statistics, University of Helsinki, Helsinki, 00014, Finland}
         \email{ting.xue@unimelb.edu.au}
\thanks{The author was supported in part by the ARC grants DP150103525.}

\dedicatory{dedicated to George Lusztig with gratitude and admiration}

\begin{abstract}
We determine character sheaves for symmetric pairs associated to spin groups. In particular, we determine the cupsidal character sheaves and show that they can be obtained via the nearby cycle construction of~\cite{GVX} and its generalisation in~\cite{VX2}.
\end{abstract}
\maketitle
\section{Introduction}

In~\cite{VX} we give a classification of character sheaves for classical symmetric pairs, which can be viewed as an analogue of Lusztig's generalized Springer correspondence~\cite{L1}. We show that all character sheaves can be obtained using the nearby cycle construction of~\cite{GVX} and parabolic induction.  As in Lusztig's generalised Springer correspondence, for symmetric pairs associated to groups in the other isogeny classes, such as special linear groups and spin groups, we need extra work to determine the character sheaves at various central characters. The inner involutions for special linear groups are treated in~\cite{VX2}. In this note we classify character sheaves for symmetric pairs associated to spin groups. In~\cite{VX3} we explain how the general reductive case can be reduced to the case of almost simple simply connected groups $G$. Thus this completes the classification of character sheaves for symmetric pairs associated to groups of classical types.

We recall the set-up in~\cite{VX}. Let $G$ be a connected complex reductive algebraic group and $\theta:G\to G$ an involution. Let $K=G^\theta$ (or $(G^\theta)^0$), and let $\Lg=\Lg_0\oplus\Lg_1$ be the decomposition of the Lie algebra $\Lg=\on{Lie}G$ induced by $\theta$ such that $d\theta|_{\Lg_i}=(-1)^i$. Let $\cA_K(\Lg_1)$ denote the set of   nilpotent orbital complexes, that is, the simple $K$-equivariant perverse sheaves on $\cN_1=\cN\cap\Lg_1$, where $\cN$ denotes the nilpotent cone of $\Lg$. By definition, the set of character sheaves $\Char_K(\Lg_1)$ for the symmetric pair $(G,K)$ consists of Fourier transforms of sheaves in $\cA_K(\Lg_1)$ (we identify $\Lg_1$ with $\Lg_1^*$).  The set $\Char_K(\Lg_1)$ is determined explicitly for all classical symmetric pairs in~\cite{VX} (where $K=G^\theta$). The symmetric pairs $(SL_n,SO_n)$  are treated in~\cite{CVX} and the pairs $(SL_{2n},Sp_{2n})$ have been studied previously in~\cite{G2,H,L}. 

In this note we will focus on involutions of spin groups $Spin_N$, $N\geq 5$. As in~\cite{VX,VX2}, we determine the character sheaves by writing down explicitly the supports of the IC sheaves and the corresponding $K$-equivariant local systems. The supports are dual strata $\widecheck\cO$ associated to nilpotent $K$-orbits $\cO$ in $\cN_1$. The local systems are given by irreducible representations of the equivariant fundamental groups $\pi_1^K(\widecheck\cO)$, which are (extended) braid groups. We write down these irreducible representations explicitly. They are given by irreducible representations of Hecke algebras associated to finite Coxeter groups with parameters $\pm1$. In particular, we determine the cuspidal character sheaves (see Theorem~\ref{theorem-cusp} below and Corollary~\ref{coro-cuspidal}), that is, the character sheaves which do not arise as a direct summand (up to shift) of parabolic induction of  character sheaves of a $\theta$-stable Levi subgroup contained in a proper $\theta$-stable parabolic subgroup. We show that they all arise from the nearby cycle construction of~\cite{GVX} and its generalisation given in~\cite{VX2}.

Let 
$
\pi:G=Spin_N\to \bar G=SO_N
$
denote the double covering homomorphism. We have a natural partition of the character sheaves by their central characters, in particular, by the action of $\ker\pi\subset K$. That is, we have  $\on{Char}_K(\Lg_1)=\on{Char}_K(\Lg_1)_{\kappa_0}\sqcup \on{Char}_K(\Lg_1)_{\kappa_1}$, where $\kappa_0$ (resp. $\kappa_1$) denote the trivial (resp. nontrivial) character of $\ker\pi\cong\bZ/2\bZ$. The set $\on{Char}_K(\Lg_1)_{\kappa_0}$ can be identified with the set $\on{Char}_{\bar K}(\Lg_1)$, where $\bar K=\pi(K)$. We determine this set following the strategy of~\cite{VX}. Note that in~\cite{VX} we work with a disconnected $\widetilde K$ and for the purpose of this paper we need to work with $\widetilde K^0=\bar K$. To determine the set  $\on{Char}_K(\Lg_1)_{\kappa_1}$ we make use of a generalisation of the nearby cycle construction as in~\cite{VX2}. 

Let $\on{Char}^{\on{cusp}}_{ K}(\Lg_1)$ (resp. $\on{Char}^{\rf}_K(\Lg_1)$, $\on{Char}^\rn_{ K}(\Lg_1)$) denote the subset of $\on{Char}_{ K}(\Lg_1)$ consisting of cuspidal (resp. full support, nilpotent support) character sheaves.  Let
$\on{Char}^{\on{cusp}}_K(\Lg_1)_{\kappa_i}=\on{Char}_K^{^{\on{cusp}}}(\Lg_1)\cap\on{Char}_K(\Lg_1)_{\kappa_i}$, $i=0,1$. Similarly for $\on{Char}_K^\rf(\Lg_1)_{\kappa_i}$, $\on{Char}_K^{\rn}(\Lg_1)_{\kappa_i}$. We show that $\on{Char}_K^{\on{cusp}}(\Lg_1)_{\kappa_0}$ is nonempty if and only if $(G,K)$ is a split symmetric pair, and $\on{Char}_K^{\on{cusp}}(\Lg_1)_{\kappa_1}$ is nonempty
when $(G,K)$ is of type BDI, that is, $\bar K\cong SO_p\times SO_q$, and $N\geq (p-q)^2$. More precisely, we have
\begin{theorem}\label{theorem-cusp}
The cuspidal character sheaves are
  \bern
&{\rm(i)}& \Char_{K^{n+t,n}}^{\on{cusp}}(\Lg_1)_{\kappa_0}=\Char_{K^{n+t,n}}^{\rf}(\Lg_1)_{\kappa_0}=\left\{\on{IC}(\Lg_1^{rs},\cT_{\psi})\,|\,\,\psi\in\Theta_{n,t}^{\kappa_0}\right\},\,|t|\leq 1
\\
&{\rm(ii)}& \on{Char}_{K^{m+\frac{t^2+t}{2},m+\frac{t^2-t}{2}}}^{\on{cusp}}(\Lg_1)_{\kappa_1}=\left\{\on{IC}(\widecheck\cO_{1^{m}_+1^m_-\sqcup{\mu_t}},\cF_{\rho})\mid\,\rho\in\Theta_{m,t}^{\kappa_1}\right\},\text{ any $t$}.
\eern
\end{theorem}
Here $K^{p,q}$ indicates  that $\pi(K)\cong SO_{p}\times SO_q$. We refer the readers to the main text for notations in the above theorem. The sheaves in (i) are obtained using the nearby cycle construction and those in (ii) using the generalised nearby cycle construction.

The paper is organised as follows. In Section~\ref{sec-0} we discuss the symmetric pairs, the nilpotent orbits and the component groups of their centralisers. We also fix some notations that will be used throughout the paper. In Section~\ref{sec-nearby} we construct a set of character sheaves using the nearby cycle construction of~\cite{GVX,GVX2} and its generalisation in~\cite{VX2}.
In Section~\ref{sec-cs} we describe the character sheaves explicitly (Theorems~\ref{thm-type BD} and~\ref{thm-type DIII}). In particular, we write down the set of nilpotent orbits $\cO$ such that the corresponding dual strata $\widecheck\cO$ are the supports of character sheaves. We determine the cuspidal character sheaves (Corollary~\ref{coro-cuspidal}), the nilpotent support character sheaves (Corollary~\ref{nil coro-2}), and their numbers. In Section~\ref{sec-proof} we prove Theorem~\ref{thm-type BD} and its corollaries. In particular, we determine the number of character sheaves (Proposition~\ref{number1}).

{\bf Acknowledgement.} I would like to thank George Lusztig for his encouragement over the years and particularly for encouraging me to work on spin groups. I would also like to thank Cheng-Chiang Tsai, Kari Vilonen and Zhiwei Yun for many helpful discussions and  Dennis Stanton for the combinatorics I have learned from him.

\section{Preliminaries and notations}\label{sec-0}

Throughout the paper let $G=Spin_N$ and $\bar G=SO_N$, $N\geq 5$. We first recall the definition of spin groups using Clifford algebras (see for example~\cite[\S 14.3]{L1}). Let $V$ be a $\bC$-vector space of dimension $N$ equipped with a non-degenerate bilinear form $(,)$. Let $C_V$ be the corresponding Clifford algebra; it can be defined as $T_V/\langle vv'+v'v=2(v,v'),\,v,v'\in V\rangle$, where $T_V$ is the tensor algebra of $V$. Let $C^+_V\subset C_V$ denote the sub-algebra  spanned by elements of the form $v_1v_2\cdots v_{2a}$, $v_i\in V$, that is, products of an even number of vectors in $V$. The spin group $G=Spin_V=Spin_N$ is the subgroup of the group of units in $C^+_V$ consisting of elements of the form $v_1v_2\cdots v_{2a}$, $a\in\mathbb{N}$, $v_i\in V$, and $(v_i,v_i)=1$. The homomorphism $\pi:Spin_V\to SO_V,\,x\mapsto (v\mapsto xvx^{-1})$ realises $G$ as a simply connected double cover of $\bar G$.
We have $\ker\pi=\{1,\epsilon\}$, where $\epsilon$ denotes the element (-1) times the unit element of $C_V^+$. 

The symmetric pairs associated to $G$ are as follows. 

(Type BD{\rm I}) Let $V=V^+\oplus V^-$ be an orthogonal decomposition such that $\dim V^+=p$ and $\dim V^-=q=N-p$.  Let $\theta:G\to G$ be an involution  such that $\pi(G^\theta)=SO_{V^+}\times SO_{V^-}$. We have
\beqn
K\cong (Spin_{V^+}\times Spin_{V^-})/\langle (\epsilon_p,\epsilon_q)\rangle:=K^{p,q},
\eeqn
where $\epsilon_p$ (resp. $\epsilon_q$) denotes the element ($-1$) times the unit element in $C^+_{V^+}$ (resp. $C^+_{V^-}$). Note that
$\epsilon=\overline{(\epsilon_p,1)}=\overline{(1,\epsilon_q)}\in K.
$
We write  $$\bar{K}^{p,q}=SO_p\times SO_q,\,\,\widetilde{K}^{p,q}=S(O_p\times O_q),\ n=[N/2],\ \ t=p-q,$$
where $[N/2]$ denotes the integer part of $N/2$. 

(Type D{\rm III}) Let $N=2n$ and let $V=V^+\oplus V^-$ be a decomposition such that $\dim V^+=\dim V^-=n$ and $(,)|_{V^+}=(,)|_{V^-}=0$.  Let $\theta:G\to G$ be an involution  such that $\bar K:=\pi(G^\theta)=\bar G\cap(GL_{V^+}\times GL_{V^-})\cong GL_n$. We have $\ker\pi\subset K$ as $\theta$ is inner.

 The group $K=G^\theta$ is connected since $G$ is simply connected. Note that $\theta$ is an {\em outer} involution if and only if it is of type BDI with both $p$ and $q$ odd, and $(G,K)$ is a split symmetric pair if and only if it is of type BDI and $|p-q|\leq 1$.

Recall $\Lg_i$, $i=0,1$, denote the $(-1)^i$-eigenspace of $d\theta$ and $\cN_1=\cN\cap\Lg_1$, where $\cN$ is the nilpotent cone of $\Lg=\on{Lie}G$. To ease the notation, we will identify $\on{Lie}\bar G$ with $\Lg$, $d\pi(\Lg_1),\,d\pi(\cN_1)$ with $\Lg_1,\,\cN_1$.
Recall also
$
\kappa_0\text{ (resp. $\kappa_1$)}:\ker\pi\to\bG_m,\ \epsilon\mapsto 1 \text{ (resp. $-1$)} 
$ 
denote the trivial (resp. nontrivial) character of $\ker\pi$.

For a finite group $H$, we write $\widehat{H}$ for the set of irreducible representations of $H$ over $\bC$ (up to isomorphism). If there is a natural map $\ker\pi\to Z(H)$, we write $\widehat{H}_{\kappa_0}$ (resp. $\widehat{H}_{\kappa_1}$) for the subset of $\widehat{H}$ consisting of the representations such that $\epsilon\in\ker\pi$ acts by $1$ (resp. $-1$), where $Z(H)$ denotes the center of $H$.

We denote by $\cP(n)$ (resp. $\cP_2(n)$) the set of partitions of $n$ (resp. bi-partitions of $n$) and let $\mathbf{p}(n)=|\cP(n)|$. By definition $\cP(x)=\cP_2(x)=\emptyset$ for $x\notin\bN$.

\subsection{Nilpotent orbits and component groups } In this subsection we describe the  $K$-orbits in $\cN_1$, the component groups $A_K(x)=Z_K(x)/Z_K(x)^0$ for $x\in\cN_1$, and $\widehat{A_K(x)}$ following~\cite[\S 14]{L1}. Let $A_{\bar K}( x)=Z_{\bar K}( x)/Z_{\bar K}( x)^0$. We can identify $\widehat{A_K(x)}_{\kappa_0}$ with $\widehat{A_{\bar K}(x)}$ via the natural projection map $A_K(x)\to A_{\bar K}(x)$ induced by $\pi$.

The nilpotent $K$-orbits in $\cN_1$ are parametrized in the same way as the nilpotent $\bar{K}$-orbits in ${\cN}_1$, see, for example, \cite{CM} and \cite{SS}. We use the same notations as in~\cite{VX}. We write a signed Young diagram as follows
\begin{subequations} 
\beq\label{signed Young diagram}
\lambda=(\lambda_1)^{p_1}_+(\lambda_1)^{q_1}_-(\lambda_2)^{p_2}_+(\lambda_2)^{q_2}_-\cdots(\lambda_s)^{p_s}_+(\lambda_s)^{q_s}_-,
\eeq
where $\lambda=(\lambda_1)^{p_1+q_1}(\lambda_2)^{p_2+q_2}\cdots(\lambda_s)^{p_s+q_s}$ is the corresponding partition of $N$, $\lambda_1>\lambda_2>\cdots>\lambda_s>0$, for $i=1,\ldots, s$, $p_i+q_i>0$ is the multiplicity of $\lambda_i$ in $\lambda$, and  $p_i\geq 0$ (resp. $q_i\geq 0$) is the number of rows of length $\lambda_i$ that begins with sign $+$ (resp. $-$). We will  sometimes replace the subscript $+$ by $0$ and $-$ by $1$ and write the signed Young diagram in~\eqref{signed Young diagram} as
\beq\label{signed Young diagram-2}
\lambda=(\lambda_1)^{p_1}_0(\lambda_1)^{q_1}_1(\lambda_2)^{p_2}_0(\lambda_2)^{q_2}_1\cdots(\lambda_s)^{p_s}_0(\lambda_s)^{q_s}_1.
\eeq 
\end{subequations}

\subsubsection{Type {\rm BDI}}Let  $\Sigma$ denote the set of signed Young diagrams
\beq\label{eqn-signed YD}
\Sigma=\{\lambda=(\lambda_1)^{p_1}_+(\lambda_1)^{q_1}_-\cdots(\lambda_s)^{p_s}_+(\lambda_s)^{q_s}_-\mid \text{$p_i=q_i$ if $\lambda_i$ is even}\}.
\eeq
 Let $\Sigma^{p,q}\subset\Sigma$ denote the subset of signed Young diagrams with signature $(p,q)$.

For $\lambda\in\Sigma$ of the form~\eqref{eqn-signed YD}, we define
\beqn\label{definition of a and b}
\begin{gathered}
a_\lambda=|\{i\in[1,s]\mid\lambda_i\equiv 1\nmod 4,\,p_i>0\}|+|\{i\in[1,s]\mid\lambda_i\equiv 3\nmod 4,\,q_i>0\}|,\\
b_\lambda=|\{i\in[1,s]\mid\lambda_i\equiv 1\nmod 4,\,q_i>0\}|+|\{i\in[1,s]\mid\lambda_i\equiv 3\nmod 4,\,p_i>0\}|.
\end{gathered}
\eeqn
Note that $a_\lambda\equiv b_\lambda+1\,\nmod 2$ if $N$ is odd, and $a_\lambda\equiv b_\lambda\equiv 1\,\nmod 2$ (resp. $a_\lambda\equiv b_\lambda\equiv 0\,\nmod 2$) if $N$ is even and $\theta$ is outer (resp. inner). 
Let  
\beqn
\begin{gathered}
\Sigma_1=\{\lambda\in\Sigma\mid a_\lambda>0,b_\lambda>0\},\ \Sigma_2=\{\lambda\in\Sigma\mid a_\lambda+b_\lambda>0,\,a_\lambda b_\lambda=0\},\\ \Sigma_3=\{\lambda\in\Sigma\mid a_\lambda=b_\lambda=0\},\text{ and }\Sigma_i^{p,q}=\Sigma^{p,q}\cap\Sigma_i,\,i=1,2,3.
\end{gathered}
\eeqn
We define
\beq\label{def of rmu}
r_\lambda=a_\lambda+b_\lambda-2\text{ (resp. $a_\lambda+b_\lambda-1,\,0$)}\text{ if }\lambda\in\Sigma_1 \text{ (resp. $\Sigma_2,\Sigma_3$)}.
\eeq
 The set of $K^{p,q}$-orbits in $\cN_1$ is
\beqn
\{\cO_\lambda\mid\lambda\in\Sigma_1^{p,q}\}\sqcup\{\cO^{\delta}_\lambda\mid\lambda\in\Sigma_2^{p,q},\delta={\rm I},{\rm II}\}\sqcup\{\cO^{\delta}_\lambda\mid\lambda\in\Sigma_3^{p,q},\delta=\rm I,\rm{II},\rm{III},\rm{IV}\}.
\eeqn

Let $\lambda\in\Sigma$ and let $\cO=\cO_\lambda$ or $\cO_\lambda^\delta$ be a $K$-orbit in $\cN_1$ corresponding to $\lambda$. Let $x_\lambda\in \cO$.

\begin{lemma}\label{lemma-component gp}

{\rm (i)} We have 
$
A_{\bar K}(x_\lambda)\cong(\bZ/2\bZ)^{r_\lambda},
$ where $r_\lambda$ is defined in~\eqref{def of rmu}.

\noindent{\rm (ii)} We have
$A_K(x_\lambda)\cong A_{\bar K}(x_\lambda)$ if there exists $\lambda_i$ odd such that $p_i\geq 2$ or $q_i\geq 2$.

\noindent{\rm (iii)} Suppose that $p_i\leq 1,\ q_i\leq 1$ for each odd $\lambda_i$. Then $A_K(x_\lambda)$ is isomorphic to a central extension of $A_{\bar K}(x_\lambda)$ by $\bZ/2\bZ$. 
We have the following cases:

{\rm (a)} Suppose that $N$ is odd. The set $\widehat{A_K(x_\lambda)}_{\kappa_1}$ consists of 2 (resp. 1)  representations of dimension $2^{\frac{r_\lambda-1}{2}}$ (resp. $2^{\frac{r_\lambda}{2}}$) if $\lambda\in\Sigma_1$ (resp. $\Sigma_2$).

{\rm (b)} Suppose that $N$ is even and $\theta$ is outer. The set $\widehat{A_K(x_\lambda)}_{\kappa_1}$ consists of 1  representation of dimension $2^{\frac{r_\lambda}{2}}$.

{\rm (c)} Suppose that $N$ is even and $\theta$ is inner. The set $\widehat{A_K(x_\lambda)}_{\kappa_1}$ consists of $4$ (resp. $2$, $1$)  representations of dimension $2^{\frac{r_\lambda-2}{2}}$ (resp. $2^{\frac{r_\lambda-1}{2}}$, $2^{\frac{r_\lambda}{2}} $) if $\lambda\in\Sigma_1$ (resp. $\Sigma_2$, $\Sigma_3$).

\end{lemma}
\begin{proof}
(i) Consider the symmetric pair $(\bar G,\bar K^{p,q})=(SO_N,SO_p\times SO_q)$.  We have 
\beqn
\begin{gathered}A_{\bar K}(x_\lambda)\cong \left(S(\prod_{\substack{\lambda_i\equiv 1\\\nmod4\\p_i>0}}O_{p_i} \prod_{\substack{\lambda_i\equiv 3\\\nmod 4\\q_i>0}}O_{q_i})\times S(\prod_{\substack{\lambda_i\equiv 1\\\nmod 4\\q_i>0}}O_{q_i}\prod_{\substack{\lambda_i\equiv 3\\\nmod{4}\\p_i>0}}O_{p_i})\right)/\prod_{\lambda_i\text{ odd}}(SO_{p_i}\times SO_{q_i}).
\end{gathered}
\eeqn
Thus (i) follows.

(ii) We follow the proof in~\cite[\S 14.3]{L1}. We have $Z_K(x_\lambda)=\pi^{-1}(Z_{\bar K}(x_\lambda))$ and
$
A_K(x_\lambda)\cong A_{\bar K}(x_\lambda)\text{ if }\epsilon\in Z_K(x_\lambda)^0$, 
$A_K(x_\lambda)\text{ is a central extension of $A_{\bar K}(x_\lambda) $ by $\bZ/2\bZ$ if }\epsilon\not\in Z_K(x_\lambda)^0.
$

Without loss of generality, we assume  that $p_i\geq 2$ for some $\lambda_i=4k+1$. Then there exist two $x_\lambda$-stable subspaces $V_{4k+1}^a$, $a=1,2$, $\dim V_{4k+1}^a=4k+1$ such that $V=V_{4k+1}^1\oplus V_{4k+1}^2\oplus W$ is an $x_\lambda$-stable orthogonal decomposition. Moreover, there exists an isometry $\gamma:V_{4k+1}^1\to V_{4k+1}^2$ such that $\gamma(x_\lambda v)=x_\lambda\gamma(v)$, $v\in V_{4k+1}^1$. We can choose an orthogonal basis $e_i^{a}\in V^+,\,i\in[1,2k+1],f_i^{a}\in V^-,i\in[1,2k]$, of $V_{4k+1}^1$ such that $(e_i^1,e_i^1)=1=(f_j^1,f_j^1)$, $e_i^2=\gamma(e_i^1)$ and $f_i^2=\gamma(f_i^1)$. Let $b,c\in\bC$ be such that $b^2+c^2=1$. Let $$V_{b,c}=\on{span}\{v_{b,c,i}:=be_i^1+ce_i^2,i=\in[1,2k+1],\,w_{b,c,j}:=bf_j^1+cf_j^2,j\in[1,2k]\}.$$ Then $V_
{b,c}$ is $x_\lambda$-stable and $\{v_{b,c,i},i\in[1,2k+1],\,w_{b,c,j},j\in[1,2k]\}$ is an orthonormal basis of $V_{b,c}$. Let $\sigma^1_+=e_1^1\cdots e_{2k+1}^1$, $\sigma^1_-=f_1^1\cdots f_{2k}^1$, $\sigma_{b,c,+}=(be_1^1+ce_1^2)\cdots (be_{2k+1}^1+ce_{2k+1}^2)$, $\sigma_{b,c,-}=(bf_1^1+cf_1^2)\cdots (bf_{2k}^1+cf_{2k}^2)$, $\sigma^1=\sigma^1_+\sigma^1_-$, $\sigma_{b,c}=\sigma_{b,c,+}\sigma_{b,c,-}$ and 
$
k_{b,c}=\sigma^1\sigma_{b,c}\sigma^1\sigma_{b,c}\in K.
$
Since $v\mapsto\sigma^1v(\sigma^1)^{-1}$ restricts to $V_0=V_{4k+1}^1\oplus V_{4k+1}^2$ (resp. $V_0^\perp=W$) is $1$ (resp. $-1$),  $v\mapsto\sigma_{b,c}v\sigma_{b,c}^{-1}$ restricts to $V_{b,c}$ (resp. $V_{b,c}^\perp$) is $1$ (resp. $-1$), and $V_0$, $V_{b,c}$ are $x_\lambda$-stable, we conclude that $\pi(k_{b,c})$ commutes with $x_\lambda$. Thus $k_{b,c}\in Z_{K}( x_\lambda)$. Since $k_{1,0}=1\in Z_K(x_\lambda)^0$, and $\{(b,c)\in\bC\mid b^2+c^2=1\}$ is irreducible, we conclude that $k_{0,1}=\epsilon\in Z_K(x_\lambda)^0 $. It follows that $A_K(x_\lambda)\cong A_{\bar K}(x_\lambda)$.

(iii) Suppose that $p_i\leq 1, q_i\leq 1$ for each odd $\lambda_i$. Let
\ber
&&J_1=\{i\in[1,s]\mid\lambda_i\equiv 1\nmod 4,\,p_i>0\}\cup\{i\in[1,s]\mid\lambda_i\equiv 3\nmod 4,\,q_i>0\},\label{set 1}\\
&&J_2=\{i\in[1,s]\mid\lambda_i\equiv 1\nmod 4,\,q_i>0\}\cup\{i\in[1,s]\mid\lambda_i\equiv 3\nmod 4,\,p_i>0\}\label{set 2}.
\eer
We have an orthogonal decomposition $V=\oplus_{i\in J_1}V_{\lambda_i}\oplus_{i\in J_2}W_{\lambda_i}\oplus W$ into $x_\lambda$-stable subspaces. Suppose that $i\in J_1$. We choose an orthonormal basis $e^{i,1}_j\in V^+,j=1,\ldots, 2[\frac{\lambda_i}{4}]+1, f_{j}^{i,1}\in V^-, j=1,\ldots, 2[\frac{\lambda_i+1}{4}]$ of $V_{\lambda_i}$ and define 
$x_i=e^{i,1}_1\cdots e_{2[\frac{\lambda_i}{4}]+1}^{i,1}f_1^{i,1}\cdots f_{2[\frac{\lambda_i+1}{4}]}^{i,1}\in C_V,\,i\in J_1.$ Suppose that $i\in J_2$. We choose an orthonormal basis $e^{i,2}_j\in V^+,j=1,\ldots, 2[\frac{\lambda_i+1}{4}], f_j^{i,2}\in V^-, j=1,\ldots, 2[\frac{\lambda_i}{4}]+1$ of $W_{\lambda_i}$ and define
$
y_i=e^{i,2}_1\cdots e_{2[\frac{\lambda_i+1}{4}]}^{i,2}f_1^{i,2}\cdots f_{2[\frac{\lambda_i}{4}]+1}^{i,2}\in C_V,\,i\in J_2.
$
We have
\beqn
x_i^2=\epsilon^{\frac{\lambda_i(\lambda_i-1)}{2}},\ x_ix_{i'}=\epsilon x_{i'}x_i,\,i\neq i',\,y_i^2=\epsilon^{\frac{\lambda_i(\lambda_i-1)}{2}},\ y_iy_{i'}=\epsilon y_{i'}y_i,\,i\neq i',\ x_iy_j=\epsilon y_j x_i.
\eeqn
Let $\hat\Gamma_1$ (resp. $\hat\Gamma_2$) be the subgroup of the group of units of $C_V$ generated by $x_i$ (resp. $y_i$), $i\in J_1$, (resp. $i\in J_2$) and $\Gamma_1\subset\hat\Gamma_1$ (resp. $\Gamma_2\subset\hat\Gamma_2$) the subgroup consisting of elements that are products of an even number of $x_i$'s (resp. $y_i$'s). Then $\Gamma_1\subset K$ and $\Gamma_2\subset K$. Moreover,
\beqn
A_K(x_\lambda)\cong \Gamma_1\times \Gamma_2/\langle(\epsilon_1,\epsilon_2)\rangle,
\eeqn
where $\epsilon_1$ (resp. $\epsilon_2$) is the element (-1) times unit in $C_{\oplus_{i\in J_1}V_{\lambda_i}}$ (resp. $C_{\oplus_{i\in J_2}V_{\lambda_i}}$). By~\cite[\S 14.3]{L1}, $\Gamma_1$ (resp. $\Gamma_2$) is an central extension of $(\bZ/2\bZ)^{a_\lambda-1}$ (resp. $(\bZ/2\bZ)^{b_\lambda-1}$) by $\bZ/2\bZ$ if $a_\lambda>0$ (resp. if $b_\lambda>0$). Moreover, $\Gamma_1$ (resp. $\Gamma_2$) has $2$ irreducible representations with non-trivial $\epsilon_1$-action (resp. $\epsilon_2$-action), each of dimension $2^{\frac{a_\lambda-2}{2}}$ (resp. $2^{\frac{b_\lambda-2}{2}}$), if $a_\lambda$ (resp. $b_\lambda$) is even; it has $1$ irreducible representation with non-trivial $\epsilon_1$-action (resp. $\epsilon_2$-action) of dimension $2^{\frac{a_\lambda-1}{2}}$ (resp. $2^{\frac{b_\lambda-1}{2}}$), if $a_\lambda$ (resp. $b_\lambda$) is odd. Thus (iii) follows. 
\end{proof}

\subsubsection{Type {\rm DIII}}Let  $\Lambda$ denote the following set of signed Young diagrams 
\beq\label{eqn-signed YDiii}
\begin{gathered}
\Lambda=\{\lambda=(\lambda_1)^{p_1}_+(\lambda_1)^{q_1}_-\cdots(\lambda_s)^{p_s}_+(\lambda_s)^{q_s}_-\mid \text{$p_i=q_i$ for odd $\lambda_i$,}\\\hspace{1in}\text{ $p_i\equiv q_i\equiv0\nmod 2$  for even $\lambda_i$}\}.
\end{gathered}
\eeq
We write $\Lambda=\Lambda^{n,n}$ to indicate the signature of the Young diagrams. 
 The set of $K$-orbits in $\cN_1$ is
$
\{\cO_\lambda\mid\lambda\in\Lambda^{n,n}\}.
$

Let $\lambda\in\Lambda$ and let $\cO=\cO_\lambda$  be the $K$-orbit in $\cN_1$ corresponding to $\lambda$. Let $x_\lambda\in \cO$.

\begin{lemma}\label{lemma-component gp2}

We have 
$
A_K(x_\lambda)=A_{\bar K}(x_\lambda)=1
$ unless all $\lambda_i$ are even. In the latter case, $A_K(x_\lambda)\cong\bZ/2\bZ$, $A_{\bar{K}}(x_\lambda)=1$ and $|\widehat{A_K(x_\lambda)}_{\kappa_1}|=1$.
\end{lemma}
\begin{proof}
The proof is similar to that of Lemma~\ref{lemma-component gp}. Suppose that $\lambda$ has an odd part of size $2k+1$. We show that $\epsilon\in Z_K(x_\lambda)^0$. There exists an $x_\lambda$-stable orthogonal decomposition $V=W\oplus W^\p$ such that $W=\on{span}\{e_i,\,,f_i,\,i\in[1,2k+1]\}$, $e_i\in V^+$, $f_i\in V^-$, $(e_i,f_j)=\delta_{2k+2,i+j}$ and $x_\lambda e_i=f_i,x_{\lambda}e_{2k+2-i}=-f_{2k+2-i},\,x_\lambda f_i=e_{i+1},x_{\lambda}f_{2k+1-i}=-e_{2k+2-i}$, $i\in[1,k]$, $x_\lambda e_{k+1}=x_\lambda f_{2k+1}=0$. Let $b,c\in\bC^*$ be such that $bc=1/2$. Let $v_{b,c,i}=be_i+cf_{2k+2-i}$, $i\in[1,k+1]$, $v_{b,c,i}=bf_{i-k-1}+ce_{3k+3-i}$, $i\in[k+2,2k+1]$. We have $(v_{b,c,i},v_{b,c,j})=\delta_{i,j}$. Let $\sigma_{b,c}=v_{b,c,1}v_{b,c,2}\cdots v_{b,c,2k+1}\in C_V$. One checks that $\sigma_{b,c}e_i\sigma_{b,c}^{-1}=2b^2f_{2k+2-i}$ (resp. $2c^2f_{2k+2-i}$) if $i\in[k+2,2k+1]$ (resp. $i\in[1,k+1]$), and $\sigma_{b,c}f_i\sigma_{b,c}^{-1}=2b^2e_{2k+2-i}$ (resp. $2c^2e_{2k+2-i}$) if $i\in[k+1,2k+1]$ (resp. $i\in[1,k]$). Let $g_{b,c}=\sigma_{b_0,c_0}\sigma_{b,c}\sigma_{b_0,c_0}\sigma_{b,c}\in Spin_V$, where $b_0=c_0=1/\sqrt{2}$. It follows that $g_{b,c}e_ig_{b,c}^{-1}=4b^4e_{i}$ (resp. $4c^4e_i$) if $i\in[k+2,2k+1]$ (resp. $i\in[1,k+1]$), and $g_{b,c}f_ig_{b,c}^{-1}=4b^4f_{i}$ (resp. $4c^4f_i$) if $i\in[k+1,2k+1]$ (resp. $i\in[1,k]$). Thus $\pi(g_{b,c})\in Z_{\bar K}(x_\lambda)$. Since $\{(b,c)\in(\bC^*)^2\mid bc=1/2\}$ is irreducible, $g_{b_0,c_0}=1$ and $g_{\sqrt{-1/2},-\sqrt{-1/2}}=\epsilon$, we conclude that $\epsilon\in Z_K(x_\lambda)^0$. Thus $A_K(x_\lambda)=A_{\bar K}(x_\lambda)=1$.

Suppose that all parts $\lambda_i$ are even. In this case $\bar K^\phi\cong \prod_i (Sp_{p_i}\times Sp_{q_i})$ is simply connected.  It follows that $A_K(x_\lambda)\cong\bZ/2\bZ$ and $|\widehat{A_K(x_\lambda)}_{\kappa_1}|=1$. 
\end{proof}
\begin{corollary}\label{corollary-Diii}We have
\bern
|\on{Char}_K(\Lg_1)|=|\on{Char}_K(\Lg_1)_{\kappa_0}|=|\on{Char}_{\bar K}(\Lg_1)| \text{ when $n$ is odd}\\
|\on{Char}_K(\Lg_1)_{\kappa_0}|=|\on{Char}_{\bar K}(\Lg_1)|,\,|\on{Char}_K(\Lg_1)_{\kappa_1}|=|\cP_2(n/2)|\text{ when $n$ is even}.
\eern
\end{corollary}
\begin{proof}
Suppose $n$ is even. Then $|\on{Char}_K(\Lg_1)_{\kappa_1}|=|\cA_K(\Lg_1)_{\kappa_1}|=|\{\lambda\in\Lambda\mid\lambda_i\text{ even for all $i$}\} |$. The set of such $\lambda$'s are in bijection with $\cP_2(n/2)$ via the following map\\
$$
(2\mu_1)^{2p_1}_+(2\mu_1)^{2q_1}_-\cdots(2\mu_s)^{2p_s}_+(2\mu_s)^{2q_s}\mapsto \left((\mu_1)^{p_1}\cdots(\mu_s)^{p_s},(\mu_1)^{q_1}\cdots(\mu_s)^{q_s}\right).
$$
\end{proof}

\subsection{Representations of extended braid groups}\label{ssec-repeb}
For a finite Coxeter group $W$, let $B_W$ denote its associated braid group with $p:B_W\to W$ and let $\cH_W$ denote a Hecke algebra associated to $W$. We write $\on{Irr}\cH_{W}$ for the set of  irreducible representations of $\cH_{W}$ over $\bC$ (up to isomorphism). For $\rho\in\on{Irr}\cH_W$, we continue to write $\rho$ for the irreducible representation of $B_{W}$ obtained by pulling back $\rho$ via the surjective map $\bC[B_W]\to\cH_W$.

Let  $\widetilde B_W=B_W\ltimes H$ be an extended braid group, where $H$ is a finite group, and let $E\in\widehat{H}$. Suppose that $W_{E,0}$ is a Coxeter subgroup of $W_E=\on{Stab}_W(E)$. We write $B_{W}^{E,0}=p^{-1}(W_{E,0})$ and $\widetilde B_{W}^{E,0}=B_{W}^{E,0}\ltimes H$. Suppose that we have a surjective map $B_{W}^{E,0}\to B_{W_{E,0}}$. For $\rho\in\on{Irr}\cH_{W_{E,0}}$,  we write
\beqn
V_{\rho,E}=\bC[\widetilde B_W]\otimes_{\bC[\widetilde B_{W}^{E,0}]}(\rho\otimes E)
\eeqn
for the induced representation of $\widetilde B_W$, where $\widetilde B_{W}^{E,0}=B_{W}^{E,0}\ltimes H$ acts on $\rho\otimes E$ via the $H$ action on $E$ and the $B_{W}^{E,0}$-action on $\rho$ via $B_{W}^{E,0}\to B_{W_{E,0}}$.
Then $V_{\rho,E}$ is an irreducible representation of $\widetilde B_W$ if $W_{E,0}=W_E$. When $W_{E,0}\neq W_E$, we write $V_{\rho,E}^\delta$ for the non-isomorphic irreducible summands of $V_{\rho,E}$ as representations of $\widetilde B_W$.

\subsection{ Notations} We write $W_r$ for the Weyl group of type $B_r$ and $W_r'$ for the Weyl group of type $D_r$. 

We write $(A;x)_\infty=\prod_{s=0}^\infty(1-Ax^s)$ and $(A;x)_n=\prod_{s=0}^{n-1}(1-Ax^s)$.

We let $\mu_t$ denote the signed Young diagram
\beq\label{def-mut}
\mu_t=(2|t|-1)_{{sgn_t}}(2|t|-3)_{sgn_t}\cdots 3_{_{sgn_t}}1_{_{sgn_t}}\text{ if }t\neq0 ,\ 
\eeq
where ${sgn_t}=+ \text{ (resp. $-$) if $t>0$ (resp. $t<0$})$ and we also write $\mu_t=\emptyset\text{ if }t=0$.

We will often indicate the parameters of the Hecke algebra $\cH_W$ by writing $\cH_{W,c}$. For example, $\cH_{W,-1}$ denote the Hecke algebra associated to $W$ with parameter $-1$, that is, the Hecke algebra generated by $T_i$ ($i$ runs through a set of simple reflections in $W$) subject to braid relations plus the Hecke relations $(T_i-1)^2=0$.   Recall also the Hecke algebras $\cH_{W_k,1,-1},\cH_{W_k,-1,1},\cH_{W_k,1,1}=\bC[W_k]$ and their simple modules (see~\cite[\S 2.4]{VX})
\beq\label{heckrep1}
\on{Irr}\cH_{W_K,1,-1}=\{L_\tau\mid\tau\in\cP(k)\},\,\,\on{Irr}\cH_{W_K,1,1}=\{L_\sigma\mid\sigma\in\cP_2(k)\}.
\eeq Moreover, we have (see~\cite{AM,G})
\begin{subequations}
\beq\label{eqn-hecke1}
\sum_{n\geq 0}|\on{Irr}\cH_{W_n,-1}|\,x^n=\prod_{s\geq 1}(1+x^{2s})(1+x^s)
\eeq
\beq\label{eqn-heckeD}
\sum_{n\geq 0}|\on{Irr}\cH_{W_n',-1}|\,x^n=\frac{1}{2}\prod_{s\geq 1}(1+x^{2s-1})(1+x^s)=\frac{1}{2}\sum_{n\geq 0}|\on{Irr}\cH_{W_n,-1,1}|\,x^n.
\eeq
\end{subequations}

\section{Nearby cycle sheaves}\label{sec-nearby}

In this section we  will produce character sheaves supported on $\widecheck\cO_{m,t}$ for type BDI and full support character sheaves for type DIII, making use of the nearby cycle construction of~\cite{GVX,GVX2} and its generalisation in~\cite{VX2}. We use the notations from~\cite{VX} and refer the readers to {\rm[loc.cit]} for details. The $\widecheck\cO_{m,t}$ are the dual stratum associated to the following nilpotent orbits in $\cN_1$ (see~\cite[\S3.2]{VX})
\beqn
\cO_{m,t}:=\cO_{1^m_+1^m_-\sqcup\mu_t}, \ m=(N-|t|^2)/2,
\eeqn
where $\mu_t$ is defined in~\eqref{def-mut}. For convenience, let us define
\beq\label{eqn-mpq}
\eta_{m,t}=
2\text{ (resp. $4,\ 1$)}\text{ if $t$ is odd (resp. $m\equiv t/2\nmod2$,  $m\equiv t/2+1\nmod2$)}.
\eeq

We have the following short exact sequence 
\beqn
1\to I_{\widecheck\cO}\to\pi_1^K(\widecheck\cO)\to B_{W_{\widecheck\cO}}\to 1
\eeqn
where $I_{\widecheck\cO}\cong A_K(x)=Z_K(x)/Z_K(x)^0$, $x\in\widecheck\cO$, and $B_{W_{\widecheck\cO}}\cong B_{W_{\fa_\phi}}$ is the braid group associated to $W_{\fa_\phi}$, $\phi=(e,f,h)$ is a normal $\mathfrak{sl}_2$-triple with $e\in\cO,h\in\Lg_0$. 

Let $\fa\subset\Lg_1$ be a Cartan subspace and  let $W_\fa=N_K(\fa)/Z_K(\fa)$ be the little Weyl group. Then $W_{\fa^\phi}=W_\fa$ if $e=0$. Note also that if $|t|\leq 1$, that is, if $\theta$ is split, then $\widecheck\cO_{m,t}=\Lg_1^{rs}$, the set of regular semisimple elements of $\Lg_1$, and $\overline{\widecheck\cO_{m,t}}=\Lg_1$.

\subsection{Nearby cycles for type DIII and split type BDI}\label{sec-split} In this subsection we describe the nearby cycle sheaves when $(G,K)$ is of type DIII or of split type BDI. 

Let $I=Z_K(\fa)/Z_K(\fa)^0$ and $\chi\in\hat I$. Let $P_\chi\in\on{Perv}_K(\cN_1)$ be  the nearby cycle sheaf  defined in~\cite{GVX} (see also~\cite[\S 3.2]{VX}). We have
(\cite[Theorem 3.6]{GVX}, see also~\cite[\S 3.3]{VX})
\beqn
\mathfrak{F}P_\chi\cong\on{IC}(\Lg_1^{rs},\cM_\chi)
\eeqn
where $\cM_\chi$ is the $K$-equivariant local system on $\Lg_1^{rs}$ given by the representation $$
M_{\chi}=\bC[\widetilde{B}_{W_\fa}]\otimes_{\bC[\widetilde{B}_{W_\fa}^{\chi,0}]}(\bC_\chi\otimes\cH_{W_{\fa,\chi}^0})
$$ of $\widetilde B_{W_\fa}=\pi_1^K(\Lg_1^{rs})$. Here $\widetilde B_{W_\fa}\cong B_{W_\fa}\ltimes I$ and $W_{\fa,\chi}^0$ is a Coxeter subgroup of $W_{\fa,\chi}$.

\subsubsection{Type {\rm DIII}}\label{nearby-d3}Suppose that $(G=Spin_{2n},K)$ is of type DIII. We have
\beqn
I=1\ (\text{resp. }\bZ/2\bZ)\text{ if $n$ is odd (resp. even)},\ \ W_\fa=W_{[n/2]}.
\eeqn
Let $\chi_0$ denote the trivial character of $I$, and when $n$ is even, let $\chi_1$ denote the nontrivial character of $I$. By~\cite{GVX}, we have
\beqn
M_{\chi_0}=\cH_{W_{[n/2]},1,-1},\ \ M_{\chi_1}=\cH_{W_{n/2},1,1}\otimes\bC_{\chi_1}\cong\bC[W_{n/2}]\otimes\bC_{\chi_1}.
\eeqn
For each $\tau\in\cP([n/2])$ (resp. $\sigma\in\cP_2(n/2)$), recall the simple module $L_\rho$ (resp. $L_\tau$) of $\cH_{W_{[n/2]},1,-1}$ (resp. $\cH_{W_{[n/2]},1,1}$). Let $\cL_\tau$ (resp. $\cL_\sigma\otimes \bC_{\chi_1}$) denote the  $K$-equivariant local system on $\Lg_1^{rs}$ corresponding to the irreducible representation $L_\tau$ (resp. $L_\sigma\otimes\bC_{\chi_1}$) of $\pi_1^K(\Lg_1^{rs})=B_{W_{[n/2]}}\times I$ where $B_{W_{[n/2]}}$ acts  on $L_\tau$ (resp. $L_\sigma$) , and $I$ acts via $\chi_0$ (resp. $\chi_1$).   It follows that
\ber
&&\{\on{IC}(\Lg_1^{rs},\cL_\tau)\mid\tau\in\cP([n/2])\}\subset\on{Char}_K^\rf(\Lg_1)_{\kappa_0},\nonumber\\
\label{Diiin}&&\{\on{IC}(\Lg_1^{rs},\cL_\sigma\otimes\bC_{\chi_1})\mid\sigma\in\cP_2(n/2)\}\subset\on{Char}_K^\rf(\Lg_1)_{\kappa_1}\ (\text{when $n$ is even})
\eer
We in fact have equality in the above two equations, which will follow once we determine all character sheaves.
\subsubsection{Split type {\rm BDI}} Suppose that $\theta$ is split, i.e., $(G,K)=(Spin_N,K^{q+t,q})$, $|t|\leq1$.  Recall that
 \bern
 &&I=\langle\gamma_1,\ldots,\gamma_n\rangle\cong(\bZ/2\bZ)^n,\ \gamma_i=\check\alpha_i(-1),\,i=1,\ldots,n,\\
 &&\text{$W_\fa=W_{n}\text{ (resp. $W_{n}'$ ) if $N$ is odd (resp. even)}$}
 \eern
 where $\check\alpha_1,\ldots,\check\alpha_n $,  is a set of simple coroots in $\check R(G,A)$ with respect to the $\theta$-split maximal torus $A=Z_G(\fa)$. 
 
 Suppose that $N=2n+1$ (resp. $N=2n$). We choose a set of simple roots as $\alpha_i=e_i-e_{i+1}$, $i=1,\ldots,n-1$, and $\alpha_n=e_n$ (resp. $\alpha_n=e_{n-1}+e_n$). We define $\chi_m\in\hat I$, $0\leq m\leq n$, by
\beqn\chi_m(\gamma_i)=1,\ i\neq m,\ \chi_m(\gamma_m)=-1.
\eeqn
Let us write
\beq\label{widetb}
\widetilde{B}_{W_n}=\widetilde{B}_{W_\fa}\cong B_{W_n}\ltimes (\bZ/2\bZ)^n\,\text{ (resp. $\widetilde{B}_{W_n'}=\widetilde{B}_{W_\fa}\cong B_{W_n'}\ltimes (\bZ/2\bZ)^n$)}.
\eeq
The following proposition can be checked directly or be derived from~\cite{VX3}.
\begin{proposition}\label{prop-nearby}
{\rm (i)} Suppose that $(G,K)=(Spin_{2n+1}, K^{q+t,q})$, $|t|=1$.  We have
\begin{eqnarray*}
&&\{\cF(P_\chi)\mid\chi\in\hat{I}/W_n\}=\{\on{IC}(\Lg_1^{rs},\cM_{\chi_m}),\, m\in[0, [\frac{n}{2}]]\}\cup\{\on{IC}(\Lg_1^{rs},\cM_{\chi_n})\},\\
&&M_{\chi_m}\cong\bC[\widetilde B_{W_n}]\otimes_{\bC[\widetilde B_{W_n}^{\chi_m,0}]}\left(\bC_{\chi_m}\otimes(\cH_{W_{m},-1}\otimes\cH_{W_{n-m},-1})\right),\text{ if } 0\leq m\leq [{n}/{2}]\\
&&M_{\chi_n}\cong\bC[\widetilde B_{W_n}]\otimes_{\bC[\widetilde B_{W_n}^{\chi_n}]}\left(\bC_{\chi_n}\otimes\cH_{S_{n},-1}\right).
\end{eqnarray*}

{\rm (ii)} Suppose that $(G,K)=(Spin_{2n}, K^{n,n})$, and $n\geq 3$. We have
\bern
&&\{\cF(P_\chi)\mid\chi\in\hat{I}/W_n'\}=\{\on{IC}(\Lg_1^{rs},\cM_{\chi_m}),\, m\in[0, [\frac{n}{2}]]\cup\{n\}\cup\,(\text{if $n$ even}) \{n-1\}\},
\\
&&M_{\chi_m}\cong\bC[\widetilde B_{W_n'}]\otimes_{\bC[\widetilde B_{W_n'}^{\chi_m,0}]}\left(\bC_{\chi_m}\otimes(\cH_{W_{m}',-1}\otimes\cH_{W_{n-m}',-1})\right),\text{ if } 0\leq m\leq [{n}/{2}]\\
&&M_{\chi_m}\cong\bC[\widetilde B_{W_n'}]\otimes_{\bC[\widetilde B_{W_n'}^{\chi_m,0}]}\left(\bC_{\chi_m}\otimes\cH_{S_{n},-1}\right),\text{ if $m=n-1,n$}.
\end{eqnarray*}
\end{proposition}
Let us write $\Theta_{n,t}$ ($|t|\leq 1$) for the set of simple $\bC[\widetilde{ B}_{W_\fa}]$-modules that appear as composition factors of $M_{\chi}$, $\chi\in\widehat I$. 
We have $\Theta_{n,t}=\Theta^{\kappa_0}_{n,t}\sqcup\Theta^{\kappa_1}_{n,t}$, where $\Theta^{\kappa_0}_{n,t}$ (resp.  $\Theta^{\kappa_1}_{n,t}$) denote the subset of modules such that $\epsilon$ acts by  $1$ (resp. $-1$) via the natural map $\ker\pi\to I$.  Note that $\epsilon=\gamma_n$ if $N$ is odd, and $\epsilon=\gamma_{n-1}\gamma_n$ if $N$ is even. Moreover, $\Theta_{n,1}=\Theta_{n,-1}$.

For each $\phi\in\Theta_{n,t}$, we write $\cT_\phi$ for the corresponding local system on $\Lg_1^{rs}$. It follows from Proposition~\ref{prop-nearby} that $\IC(\Lg_1^{rs},\cT_\phi)\in\on{Char}_{K}^\rf(\Lg_1)$.  Suppose that $N$ is odd. We have $W_{\fa,\chi_m}=W_{\fa,\chi_m}^0$ for $0\leq m<n/2$ and
$W_{\fa,\chi_{n/2}}/ W_{\fa,\chi_{n/2}}^0\cong\bZ/2\bZ$. 
 Suppose that $N$ is even. We have $W_{\fa,\chi_m}=W_{\fa,\chi_m}^0\text{ if }m=0,\text{ or if $m=n$ and $n$ is odd}$,  $W_{\fa,\chi_m}/ W_{\fa,\chi_m}^0\cong\bZ/2\bZ,\,\text{ if $1\leq m<n/{2}$ or if }m=n-1,n\text{ and $n$ is even},\ W_{\fa,\chi_{\frac{n}{2}}}/W_{\fa,\chi_{\frac{n}{2}}}^0\cong\bZ/2\bZ\times\bZ/2\bZ\,.$ 
Applying the discussion in \S\ref{ssec-repeb}, we conclude that the IC sheaves in the following proposition are full support character sheaves.   The fact that these are all full support character sheaves will follow once we determine all character sheaves.

\begin{proposition}\label{prop-full}Suppose that $(G,K)=(Spin_{N},K^{q+t,q})$, where $|t|\leq1$.
We have
\ber
\label{Theta0}&&\on{Char}^\rf_K(\Lg_1)_{\kappa_i}=\{\on{IC}(\Lg_1^{rs}, \cT_{\rho})\mid\rho\in\Theta^{\kappa_i}_{n,|t|}\},\,i=0,1,\nonumber
\\
&&\Theta^{\kappa_0}_{n,1}=\{ V_{\rho_1\otimes\rho_2, \chi_k}\mid\rho_1\in\on{Irr}\cH_{W_k,-1},\,\rho_2\in\on{Irr}\cH_{W_{n-k},-1},\, k\in[0,n/2],\rho_1\not\cong\rho_2\}, \\
&&\cup\{V_{\rho\otimes\rho,\chi_{n/2}}^\delta\mid\rho\in\on{Irr}\cH_{W_{n/2},-1},\,\delta=\rm I,\rm{II}\}\text{ where $V_{\rho_1\otimes\rho_2, \chi_{n/2}}\cong V_{\rho_2\otimes\rho_1, \chi_{n/2}}$,}\nonumber\\
&&\Theta^{\kappa_0}_{n,0}=\{V_{\rho_1\otimes\rho_2, \chi_m}^\delta\mid\rho_1\in\on{Irr}\cH_{W_m',-1},\,\rho_2\in\on{Irr}\cH_{W_{n-m}',-1},\, m\in[1,\frac{n}{2}],\rho_1\not\cong\rho_2,\,\delta=\rm{I,II}\}\nonumber\\
&&\cup\{ V_{\rho\otimes\rho, \chi_{n/2}}^\delta\mid\rho\in\on{Irr}\cH_{W_{n/2}',-1},\,\delta={\rm I,II,III,IV}\}\cup\{V_{\rho,\chi_{0}}\mid\rho\in\on{Irr}\cH_{W_{n}',-1}\},\nonumber\\
&&\text{ where $V^\delta_{\rho_1\otimes\rho_2, \chi_{n/2}}\cong V^\delta_{\rho_2\otimes\rho_1, \chi_{n/2}}$,}\nonumber\\
\label{Thetant}&&\Theta^{\kappa_1}_{n,1}=\{V_{\rho,\chi_{n}}^\delta\mid\rho\in\on{Irr}\cH_{S_n,-1},\,\delta={\rm I,II}\},\ \Theta^{\kappa_1}_{n,0}=\{V_{\rho,\chi_{n}}\mid\rho\in\on{Irr}\cH_{S_n,-1}\}\text{ if $n$ is odd},\\
&&\Theta^{\kappa_1}_{n,0}=\{V_{\rho,\chi_{i}}^\delta\mid\rho\in\on{Irr}\cH_{S_n,-1},\,\delta={\rm I},{\rm II}, i=n-1,n\}\text{ if $n$ is even}.\nonumber
\eer
\end{proposition}

\subsection{Generalised nearby cycles}\label{sec-cuspn}
In this subsection we assume that $(G,K)$ is of type BDI and $ \widecheck\cO=\widecheck\cO_{1^m_+1^m_-\sqcup\mu_t}$, $|t|\geq 2$. Let us write $I_{\widecheck\cO}:= I_{m,t}$.

\begin{lemma}\label{lemma-gnc}
{\rm (i)} We have
\beqn
 I_{m,t}\cong
A_K(\cO_{\mu_t})\text{ if $m=0$} ,\ I_{m,t}\cong
\Gamma_{|t|}\times(\bZ/2\bZ)^{m-1}\text{ if $m\geq 1$}\,,
\eeqn
where $\Gamma_{|t|}$ is a central extension of $(\bZ/2\bZ)^{|t|-1}$ by $\bZ/2\bZ$. Moreover, when $m\geq 1$, $(\widehat{I}_{m,t})_{\kappa_1}$ consists of $2^{m-1}$ (resp. $2^m$) irreducible representations of dimension $2^{(|t|-1)/2}$ (resp. $2^{(|t|-2)/2}$) if $t$ is odd (resp. even). 

{\rm (ii)}
Suppose that $m\geq 1$. The action of ${W_{\widecheck\cO}}={W_m}$ on $(\widehat{I_{\widecheck\cO}})_{\kappa_1}$ is transitive (resp. has two orbits) if $\eta_{m,t}\in\{1,2\}$ (resp. $\eta_{m,t}=4$), where $\eta_{m,t}$ is defined in~\eqref{eqn-mpq}. More precisely, let $E$ (resp. $E_1,E_2$) be the irreducible representation(s) in $(\widehat{I_{\widecheck\cO}})_{\kappa_1}$  such that the factor $(\bZ/2\bZ)^{m-1}$ acts trivially. We have $(\widehat{I_{\widecheck\cO}})_{\kappa_1}={W_m}E$ (resp. $(\widehat{I_{\widecheck\cO}})_{\kappa_1}={W_m}E_1\sqcup {W_m}E_2$).

 Moreover, $\on{Stab}_{W_m}(E)\cong S_m$ if $\eta_{m,t}=1$, $\on{Stab}_{W_m}(E)\cong S_m\rtimes\langle\tau\rangle$ if $\eta_{m,t}=2$, and $\on{Stab}_{W_m}(E_i)\cong S_m\rtimes\langle\tau\rangle$ if $\eta_{m,t}=4$, $i=1,2$, where $\tau^2=1$.

\end{lemma}

\begin{proof}The lemma holds when $m=0$ since we have $\widecheck\cO=\cO$. Suppose that $m\geq 1$.    Let $x\in \widecheck\cO$. We have an orthogonal decomposition $V=\oplus_{i=1}^{[\frac{|t|+1}{2}]}U_{i}\oplus_{i=1}^{[\frac{|t|}{2}]}W_{i}\oplus\oplus_{i=1}^m V_i$ into $x$-stable subspaces such that $\dim U_i=4i-3$, $\dim W_i=4i-1$ and $\dim V_i=2$. Moreover, there exist an orthonormal basis $e^{i,1}_j\in V^{sgn_t},j\in[1, 2i-1], f_{j}^{i,1}\in V^{-sgn_t}, j\in[1, 2i-2]$ of $U_i$, an orthonormal basis $e^{i,2}_j\in V^{sgn_t},j\in[1, 2i], f_{j}^{i,2}\in V^{-sgn_t}, j\in[1,2i-1]$ of $W_i$ and an orthonormal basis $e^{i,3}\in V^+,\,f^{i,3}\in V^-$ of $V_i$. We define  
$x_i=e^{i,1}_1\cdots e_{2i-1}^{i,1}f_1^{i,1}\cdots f_{2i-2}^{i,1}\in C_V$, $i\in[1,[\frac{|t|+1}{2}]]$, $
y_i=e^{i,2}_1\cdots e_{2i}^{i,2}f_1^{i,2}\cdots f_{2i-1}^{i,2}e^{1,3}f^{1,3}\in C_V,\,i\in[1,[\frac{|t|}{2}]]
$, and $z_i=e^{i,3}f^{i,3}\in C_V$, $i\in[1, m]$. Let $\Gamma_{|t|}$ be the subgroup of $C_V^{\times}$ consisting of the products of an even number of $x_i$, $y_i$'s (combined). Then $A_K(x)\cong\Gamma_{|t|}\times\langle \gamma_i,i\in[1,m-1]\rangle$, where $\gamma_i=z_iz_{i+1}$ and $\langle \gamma_i,i\in[1,m-1]\rangle\cong(\bZ/2\bZ)^{m-1}$. Moreover, $\Gamma_{|t|}$ is a central extension of $(\bZ/2\bZ)^{|t|-1}$ by $\bZ/2\bZ=\{1,\epsilon\}$ and $\bC[\Gamma_{{|t|}}]/(\epsilon+1)$ is isomorphic to the positive part of the Clifford algebra in $|t|$ variables, $C_{|t|}^+$. Part (i) of the lemma follows.

To prove part (ii) of the lemma, we note that when $|t|$ is even, the two non-isomorphic irreducible representations of 
$C_{|t|}^+$ are distinguished by the action of $\gamma=x_1\cdots x_{\frac{|t|}{2}}y_1\cdots y_{\frac{|t|}{2}}$. Let $\rho\in(\widehat{I_{\widecheck\cO}})_{\kappa_1}$. For simplicity, let us continue to write $\gamma_j$, $\gamma$ in place of $\rho(\gamma_j),\rho(\gamma)$. One checks that the action of ${W_m}=\langle s_1,\ldots,s_m\rangle$ (where $s_m$ is the special simple reflection) are as follows
 $s_i(\gamma_j)=\gamma_j,\,j\neq i-1,i+1,\ s_i(\gamma_j)=\gamma_i\gamma_j,\,j=i-1,i+1,\,i\in[1,m-1],\,s_m\gamma_j=\gamma_j,\,j\neq m-1,\,s_m(\gamma_{m-1})=-\gamma_{m-1}$. Moreover, when $t\equiv 0\,\nmod 4 $,  
$\ s_i\gamma=\gamma,\,i\neq m,\,s_m\gamma=-\gamma,$ and when $t\equiv 2\,\nmod4$, 
$s_i\gamma=\gamma,\,i\neq 1,m,\,s_1\gamma=\gamma\gamma_1,\,s_m\gamma=-\gamma.$ It follows that $\on{Stab}_{W_m}E,\on{Stab}_{W_m}{E_i}\subset\on{Stab}_{W_m}\chi=\langle s_1,\ldots,s_{m-1}\rangle\rtimes\langle\tau\rangle$, where $\chi:\langle\epsilon,\gamma_1,\ldots,\gamma_{m-1})\to\bC^*$, $\chi(\epsilon)=-1$, $\chi(\gamma_i)=1$ and $\tau=\prod_{i=1}^{[\frac{m+1}{2}]}s_{e_i+e_{m+1-i}}$. Here $s_{e_i+e_{m+1-i}}=s_{e_i}s_{e_i-e_{m+1-i}}s_{e_i}$ if $i\neq (m+1)/2$, $s_{e_i}=s_is_{i+1}\cdots s_{m-1}s_ms_{m-1}\cdots s_{i+1}s_i$, and $s_{e_i-e_{m+1-i}}=s_is_{i+1}\cdots s_{m-i-1}s_{m-i}s_{m-i-1}\cdots s_{i+1}s_i$. Furthermore, when $t\equiv 0\,\nmod4$,  $\tau\gamma=-\gamma$ (resp. $\gamma$) if $m$ is odd (resp. even); and when $t\equiv 2\,\nmod4$,  $\tau\gamma=\gamma\prod_{j=1}^{m-1}\gamma_j$ (resp. $-\gamma\prod_{j=1}^{m-1}\gamma_j$) if $m$ is odd (resp. even).  One then readily verifies that $\on{Stab}_{W_m}E$ and $\on{Stab}_{W_m}E_i$ are as desired and that  $E_1$ and $E_2$ are not in the same $W_m$-orbit. This completes the proof of the lemma.
\end{proof}

Suppose that $m\geq 1$ and $|t|\geq 2$. Let $\cE$ (resp. $\cE_1,\cE_2$) be the $K$-equivariant local system(s) on $X_{x}=K.x$, $x\in\widecheck\cO$, corresponding to the irreducible representation(s) $E$ (resp. $E_1,E_2$) in $(\widehat{I_{\widecheck\cO}})_{\kappa_1}$  when $\eta_{m,t}\in\{1,2\}$ (resp. when $\eta_{m,t}=4$). As in~\cite{VX2}, we form the nearby cycle sheaf  $P_{\widecheck\cO,\cE}\in\on{Perv}_{K}(\cN_1)$ (resp. $P_{\widecheck\cO,\cE_i}\in\on{Perv}_{K}(\cN_1)$, $i=1,2$).

Let us write $\widetilde{B}_{m,t}=\pi_1^K(\widecheck\cO_{m,t})=B_{W_m}\ltimes I_{m,t}$.  We use the notations from \S\ref{ssec-repeb}.

\begin{prop}\label{nearby-nontrivial} We have
$\mathfrak{F} P_{\widecheck\cO,\cE}\cong\on{IC}(\widecheck\cO,\cM_\cE)$ (resp. $\mathfrak{F} P_{\widecheck\cO,\cE_i}\cong\on{IC}(\widecheck\cO,\cM_{\cE_i})$, $i=1,2$), where $\cM_\cE$ (resp. $\cM_{\cE_i}$) corresponds to the following representation of $\pi_1^K(\widecheck\cO)$
\bern
M_E=\bC[\widetilde B_{m,t}]\otimes_{\bC[\widetilde B_{m,t}^{E,0}]}(\cH_{S_m,-1}\otimes E)\text{ (resp. $M_{E_i}=\bC[\widetilde B_{m,t}]\otimes_{\bC[\widetilde B_{m,t}^{E_i,0}]}(\cH_{S_m,-1}\otimes E_i)$)}.
\eern
\end{prop}
\begin{proof}
By Lemma~\ref{lemma-gnc},  
$W_{\widecheck\cO,E}^0\cong W_{\widecheck\cO,E_i}^0\cong S_m$. Moreover, 
$W_{\widecheck\cO,E}/W_{\widecheck\cO,E}^0=1$ (resp. $\bZ/2\bZ$) if $\eta_{m,t}=1$ (resp. 2), 
 and
$W_{\widecheck\cO,E_i}/W_{\widecheck\cO,E_i}^0\cong\bZ/2\bZ$.
The proposition follows from~\cite[Theorem 3.2]{VX2} and Proposition~\ref{prop-nearby} provided that we check Assumption 3.1 in~\cite{VX2} holds. Let $e\in \cO_{m,t}$ and $a\in(\La^{\phi})^{rs}$. As in~\cite{VX2}, it suffices to show that for any $e'\in Z_{\Lg_1}(a)\cap\cO'$, where $\cO'\subset\bar{\cO}$ and $\cO'\neq\cO$, and $a'\in(\La^{\phi_{e'}})^{rs}$, $Z_{K}(a,a',e')/Z_{K}(a,a',e')^0$ does not afford the character $\kappa_1$. We have
\bern
&&Z_G(a)\cong\left(\{(v,g_1,\ldots,g_m)\in\bC^*\times GL_1\times\cdots\times GL_1|v^2=g_1\cdots g_m\}\times Spin _{t^2}\right)/\langle\nu,\epsilon_t\rangle\\
&&Z_K(a)\cong\left(K_1\times Spin _{\frac{t^2+t}{2}}\times Spin _{\frac{t^2-t}{2}}\right)/\langle\nu,\epsilon_1,\epsilon_2\rangle
\eern
where $\nu=(-1,1,\ldots,1)$, $\epsilon_t$ is the nontrivial element in the kernel of $Spin_{t^2}\to SO_{t^2}$ and
$K_1=\{(v,g_1,\ldots,g_l)\in\bC^*\times O_1\times\cdots\times O_1|v^2=g_1\cdots g_l\}$.  
The action of $Z_K(a)$ on $Z_{\Lg_1}(a)$ is isomorphic to the action of $K'$ on $\Lg_1'$ for the symmetric pair $(G',K')=(Spin _{t^2},K^{\frac{t^2+t}{2},\frac{t^2-t}{2}})$. 
We are reduced to considering the symmetric pair $(G',K')$ and which $\check\cO'$ can afford the non-trivial character $\kappa'_1$ for any $\cO'\subset\bar\cO_{\mu_t}$. Suppose that $\cO'=(\lambda_1)^{p_1}_+(\lambda_1)_-^{q_1}\cdots(\lambda_s)^{p_s}_+(\lambda_s)_-^{q_s}$. The same argument as in Lemma~\ref{lemma-component gp} shows that $\check\cO'$ can afford $\kappa'_1$ only if $|p_i-q_i|\leq 1$. Let $n_0=|\{i|p_i=q_i\}|$, $n_1=|\{i|p_i=q_i+1\}|$, $n_{-1}=|\{i|p_i=q_i-1\}|$. Then we have $n_1-n_{-1}=t$. Thus $n_1\geq t$
\beqn
\sum(p_i+q_i)\lambda_i\geq 1+3+\cdots+2t-1=t^2.
\eeqn
Here we have used that if $p_i=q_i+1$, then $\lambda_i$ is odd. We conclude that $\cO'=\cO_{\mu_t}$. This completes the proof of the proposition.
\end{proof}
 
Let us write $\Theta_{m,t}^{\kappa_1}$ ($|t|\geq 2$) for the set of simple $\bC[\widetilde{ B}_{m,t}]$-modules that appear as composition factors of $M_{E}$, or $M_{E_i}$, $i=1,2$.  For each $\phi\in\Theta_{m,t}^{\kappa_1}$, we write $\cT_\phi$ for the corresponding local system on $\widecheck\cO_{m,t}$. We have the following corollary of Proposition~\ref{nearby-nontrivial}. Note that the proof of the proposition implies that for the symmetric pair $(Spin (t^2),K^{\frac{t^2+t}{2},\frac{t^2-t}{2}})$, the only $\widecheck\cO$ that can afford the non-trivial character $\kappa_1$ is $\widecheck\cO_{\mu_t}=\cO_{\mu_t}$.
\begin{corollary}Suppose that $(G,K)=(Spin_N, K^{q+t,q})$, $|t|\geq 2$, and $m=\frac{N-t^2}{2}$. We have
\bern
&&\left\{\IC(\widecheck\cO_{1^m_+1^m_-\sqcup\mu_t},\cT_\phi)\mid\phi\in\Theta_{m,t}^{\kappa_1}\right\}\subset\on{Char}_K(\Lg_1)_{\kappa_1},\\
&&\Theta_{0,t}^{\kappa_1}=\widehat{A_K({\mu_t})}_{\kappa_1},\ \Theta_{m,t}^{\kappa_1}=\{V_{\rho,E}^\delta\mid\rho\in\on{Irr}\cH_{S_m,-1},\,\delta={\rm I,II}\}\text{ if  $m\geq 1$ and $\eta_{m,t}=2$}\\
&&\Theta_{m,t}^{\kappa_1}=\{V_{\rho,E_i}^\delta\mid\rho\in\on{Irr}\cH_{S_m,-1},\,\delta={\rm I,II},\,i=1,2\}\text{ if $m\geq 1$ and $\eta_{m,t}=4$}\\
&&\Theta_{m,t}^{\kappa_1}=\{V_{\rho,E}\mid\rho\in\on{Irr}\cH_{S_m,-1}\}\text{ if  $m\geq 1$ and $\eta_{m,t}=1$}.
\eern

\end{corollary}

\section{Character sheaves}\label{sec-cs}
In this section we give explicit descriptions of the character sheaves following the strategy given in~\cite{VX}.

\subsection{Supports of character sheaves and equivariant fundamental groups}In this subsection we describe the supports of character sheaves. As in~\cite{VX}, we define a set $\underline{\cN_1}^{\text{cs},\kappa_i}$ of nilpotent $K$-orbits in $\cN_1$ such that for $\cO\in\underline{\cN_1}^{\text{cs},\kappa_i}$ the corresponding $\widecheck\cO$ are the supports of character sheaves in $\on{Char}_{ K}(\Lg_1)_{\kappa_i}$, $i=0,1$. 

\subsubsection{Type {\rm BDI}} Assume that $(G,K)$ is of type BDI. We first recall the Richardson orbits associated to $\theta$-stable Borel groups from~\cite{VX}.
\begin{definition} 

Let $\Sigma_b^{p,q}\subset\Sigma^{p,q}$ consist of signed Young diagrams of the form (see~\eqref{signed Young diagram-2})
\beqn
\lambda=(2\mu_1+1)_{\epsilon_1}(2\mu_2+1)_{\epsilon_2}\cdots(2\mu_k+1)_{\epsilon_k}
\eeqn
where $\mu_1\geq\mu_2\geq\cdots\geq\mu_k\geq 0$, $\epsilon_i\in\{0,1\}$ and $\epsilon_i=\epsilon_j$ if $\mu_i=\mu_j$,
 such that the following conditions hold
\begin{enumerate}
\item[{\rm (a)}] $\text{ when $N$ is odd,  $\epsilon_1\equiv \min(p,q)\ \nmod 2$ and $\epsilon_{2i}+\mu_{2i}\equiv \epsilon_{2i+1}+\mu_{2i+1}\,\nmod 2$ for $i\geq 1$;}$
\item[{\rm (b)}]$\text{ when $N$ is even,  $\epsilon_{2i-1}+\mu_{2i-1}\equiv \epsilon_{2i}+\mu_{2i}\mod 2$ for $i\geq 1$.}$
\end{enumerate}
\end{definition}
Note that $\Sigma_b^{p,q}\cap\Sigma_3=\emptyset$. Let
\beqn
\Sigma_b=\sqcup_{p,q}\Sigma_b^{p,q},\,\Sigma_{b,i}^{p,q}=\Sigma_{b}^{p,q}\cap\Sigma_i,\,\Sigma_{b,i}=\Sigma_{b}\cap\Sigma_i,\,i=1,2.
\eeqn
The set of  Richardson orbits attached to $\theta$-stable Borel subgroups is the following
\beqn
\{\cO_\lambda\mid\lambda\in\Sigma_{b,1}\}\sqcup\{\cO_\lambda^{\delta}\mid\lambda\in\Sigma_{b,2},\delta={\rm I,II}\}.
\eeqn

 The set $\underline{\cN_1}^{\text{cs},\kappa_0}$ consists of the following orbits:
\begin{eqnarray*}
 &&\cO_{m,k,\mu}=\cO_{1^m_+1^m_-2^k_+2^k_-\sqcup\mu},\ \  m\equiv q\hspace{-.1in}\mod 2\text{ if $N$ is even},\  \mu\in \Sigma_b^{p-m-2k,q-m-2k}\cup\{\emptyset\}. 
 \end{eqnarray*}
  The set $\underline{\cN_1}^{\text{cs},\kappa_1}$ consists of the following orbits:
\bern
&&\cO_{m,t}=\cO_{1^m_+1^m_-2^k_+2^k_-\sqcup\mu_t},\ \,\,k=(N-t^2-2m)/4,
\eern
where $\mu_t$ is defined in~\eqref{def-mut}. Here $\lambda\sqcup\mu$ denotes the signed Young diagram obtained by joining $\lambda$ and $\mu$ together, i.e., the rows of $\lambda\sqcup\mu$ are the rows of $\lambda$ and $\mu$ rearranged according to the lengths of the rows.  Moreover, $\cO_{m,k,\mu}$ denotes $\cO_{m,k,\mu}^\delta$ if $m=0$ and $\mu\in\Sigma_{b,2}\cup\{\emptyset\}$,  and $\cO_{m,t}$ denotes $\cO_{m,t}^\delta$ if $m=0$ and $|t|\leq1$.

 Note that $\underline{\cN_1}^{\text{cs},\kappa_0}\cap \underline{\cN_1}^{\text{cs},\kappa_1}\neq\emptyset$ if and only if $|t|\leq 1$, that is, $\theta$ is split. In the latter case $\underline{\cN_1}^{\text{cs},\kappa_1}\subset \underline{\cN_1}^{\text{cs},\kappa_0}$.

In what follows we describe the equivariant fundamental groups of the above $\widecheck\cO$'s. The character sheaves in $\Char_K(\Lg_1)_{\kappa_0}$ are given by local systems supported on $\widecheck\cO_{m,k,\mu}$ corresponding to irreducible representations of the equivariant fundamental group $\pi_1^K(\widecheck\cO_{m,k,\mu})$ for which the actions of $\pi_1^{K}(\widecheck\cO_{m,k,\mu})$ factor through $\pi_1^{\bar K}(\widecheck\cO_{m,k,\mu})$. Thus it suffices to write down  the equivariant fundamental groups $\pi_1^{\bar K}(\widecheck\cO_{m,k,\mu})$. Using entirely similar argument as in~\cite[\S6.1]{VX}, we conclude that 
\ber
\label{explicit fundamental1}
&&\pi_1^{\bar K}(\widecheck\cO_{m,k,\mu})\cong\begin{cases} {}_0\widetilde{B}_{W_m'}\times \widetilde{B}_{W_k}^1&\text{ if }\mu=\emptyset\\
\widetilde{B}_{W_m}^1\times \widetilde{B}_{W_k}^1\times(\bZ/2\bZ)^{r_\mu}&\text{ if $\mu\neq \emptyset$ and $\mu\in \Sigma_{b,1}$}\\
{}_0\widetilde{B}_{W_m}\times \widetilde{B}_{W_k}^1\times(\bZ/2\bZ)^{r_\mu}&\text{ if $\mu\neq \emptyset$ and $\mu\in \Sigma_{b,2}$}\end{cases}\\
\label{explicit fundamental2}&&\pi_1^K(\widecheck\cO_{m,t})\cong\widetilde B_{W_k}^1\times \widetilde{B}_{m,t},\ k=(N-t^2-2m)/4
\eer
 where $ {}_0\widetilde{B}_{W_m}=(\bZ/2\bZ)^{m-1}\rtimes B_{W_m},\ {}_0\widetilde{B}_{W_m'}=(\bZ/2\bZ)^{m-1}\rtimes B_{W_m'},\,\widetilde{B}_{W_m}^1=(\bZ/2\bZ)^{m}\rtimes^1 B_{W_m},$ $\widetilde{B}_{m,0}=\widetilde B_{W_m'}$, $\widetilde{B}_{m,1}=\widetilde B_{W_m}$,  $\widetilde{B}_{m,t}\cong B_{W_m}\ltimes I_{m,t}$ ($|t|\geq 2$, see \S\ref{sec-cuspn}), and $r_\mu$ are defined in~\eqref{def of rmu}. Moreover, ${}_0\widetilde{B}_{W_0}={}_0\widetilde{B}_{W_0'}=B_{W_0}=B_{W_0'}=\widetilde{B}_{W_0}^1=\{1\}$ and the superscript $1$ in the notation of $\widetilde B_{W_m}^1$ indicates the difference of the action of the braid group on the 2-group with that in $\widetilde{B}_{W_m}$ of~\eqref{widetb}. Moreover, $\widetilde B_{W_m}^1$ can be identified with  $\widetilde B_{W_m}$ of~\cite{VX}.

\subsubsection{Type {\rm DIII}}Assume that $(G,K)$ is of type DIII. We first recall the Richardson orbits associated to $\theta$-stable Borel groups from~\cite{VX}.
\begin{definition} 

Let $\Lambda_b^{n,n}\subset\Lambda^{n,n}$ (see~\eqref{eqn-signed YDiii}) consist of signed Young diagrams such that $p_i=q_i\leq 1$ when $\lambda_i$ is odd, and $p_i\cdot q_i=0$ when $\lambda_i$ is even.\end{definition}

The set of  Richardson orbits attached to $\theta$-stable Borel subgroups is
$
\{\cO_\lambda\mid\lambda\in\Lambda_b^{n,n}\}.
$ Note that by Corollary~\ref{corollary-Diii}, $\on{Char}_K(\Lg_1)_{\kappa_1}=\emptyset$ when $n$ is odd. 
We have 
\begin{eqnarray*}
 &&\underline{\cN_1}^{\text{cs},\kappa_0}=\{\cO_{k,\mu}=\cO_{1^{2k}_+1^{2k}_-\sqcup\mu}\mid 0\leq k\leq[n/2],\  \mu\in \Lambda_b^{n-2k,n-2k}\cup\{\emptyset\}\},\\
 &&\underline{\cN_1}^{\text{cs},\kappa_1}=\{\cO_{n/2,\emptyset}=\cO_{1^n_+1^n_-}\}\text{ when $n$ is even}.
 \end{eqnarray*}
Note that $\underline{\cN_1}^{\text{cs},\kappa_1}\subset \underline{\cN_1}^{\text{cs},\kappa_0}$. 
The equivariant fundamental groups are as follows 
\ber
\label{explicit fundamental3}
&&\pi_1^{\bar K}(\widecheck\cO_{k,\mu})\cong{B}_{W_k}\text{ when $\mu\neq\emptyset$},\ \ \pi_1^K(\widecheck \cO_{n/2,\emptyset})\cong{B}_{W_{n/2}}\times \bZ/2\bZ.
\eer

\subsection{The set $\Pi_{\cO}$  associated to Richardson orbits in type BDI} In this subsection we assume that $(G,K)$ is of type BDI. To write down irreducible representations of the equivariant fundamental groups that give rise to character sheaves, we define a subset  $\Pi_{\cO_\mu}\subset\widehat{A_K(\cO_\mu)}_{\kappa_0}\cong\widehat{A_{\bar K}(\cO_\mu)}$ for the Richardson orbits $\cO_\mu$ following~\cite{VX}, where $A_K(\cO_\mu):=A_K(x_\mu)$, $A_{\bar K}(\cO_\mu):=A_{\bar K}(x_\mu)$, $x_\mu\in\cO_\mu$. 
Let $\mu\in\Sigma_b$ and let $\cO_\mu$ be the corresponding Richardson orbit.    We can write $\mu$ as follows 
\beqn
\mu=(2\mu_1+1)_{\epsilon_1}^{m_1}(2\mu_2+1)^{m_2}_{\epsilon_2}\cdots(2\mu_s+1)^{m_s}_{\epsilon_s},
\eeqn
where $\mu_1>\mu_2>\cdots>\mu_s\geq 0$, $m_i>0$  and $\epsilon_i\in\{0,1\}$, $1\leq i\leq s$. We define
\beqn\label{def of Js}
\begin{gathered}J_\mu^1=\{i=1,\ldots,s\mid \mu_i\equiv 0\hspace{-.1in}\mod 2,\epsilon_i=0\}\cup\{i=1,\ldots,s\mid \mu_i\equiv1\hspace{-.1in}\mod 2,\epsilon_i=1\}\\
J_\mu^2=\{i=1,\ldots,s\mid \mu_i\equiv 0\hspace{-.1in}\mod 2,\epsilon_i=1\}\cup\{i=1,\ldots,s\mid \mu_i\equiv1\hspace{-.1in}\mod 2,\epsilon_i=0\}.
\end{gathered}
\eeqn
Let $x\in\cO_\mu$. Then
$
A_{\bar K}(\cO_\mu)\cong S\left(\prod_{i\in J_\mu^1}O_{m_i}\right)\times S\left(\prod_{i\in J_\mu^2} O_{m_i}\right)/\left(SO_{m_1}\times \cdots\times SO_{m_s}\right).
$
We define $\Omega_{\cO_\mu}=\Omega_\mu\subset\{1,\ldots,s\}$ to be the set of $j\in[1,s]$ such that 
\begin{enumerate}[topsep=-1ex]
\item $\sum_{a= j}^sm_a$ is even,
\item if $j\geq 2$, then either $\mu_{j-1}\geq \mu_{j}+2$ or $\epsilon_{j-1}= \epsilon_j$.
\end{enumerate}
We set
$l_\mu=l_{\cO_\mu}:=|\Omega_{\cO_\mu}|.
$ Suppose that $\Omega_{\cO_\mu}=\{j_1,\ldots,j_l\}$, $j_1<\cdots<j_l$, $l=l_\cO$, and we write $j_{l+1}=s+1$. Note that $j_1=1$ if and only if $N$ is even. Thus $l_\cO\geq 1$ when $N$ is even. 

For $1\leq i\leq s-1$, let $\delta_i$ denote the generator of $S(O_{m_{i}
}\times O_{m_{i+1}})/(SO_{m_{i}}\times SO_{m_{i+1}})\cong\bZ/2\bZ$. Suppose that $\mu\in\Sigma_{b,1}$. Then $J_\mu^1\neq \emptyset$ and $J_\mu^2\neq\emptyset$. Let $i_1<i_2<\ldots<i_t$ be the set of $j$'s in $[1,s-1]$ such that $\mu_{j}+\epsilon_{j}\neq\mu_{j+1}+\epsilon_{j+1}$. Then $A_{\bar K}(\cO_\mu)$ is generated by $\delta_j$, $j\in[i_a+1,i_{a+1}-1]$, $0\leq a\leq t$, $\tau_{i_a}=\prod_{b=i_a}^{i_{a+1}}\delta_b$, $a=1,\ldots,t-1$. Here we have written $i_0=0$ and $i_{t+1}=s$. Note that $i_a+1\in\Omega_\mu$, $a=1,\ldots,t$. Suppose that $\mu\in\Sigma_{b,2}$. Then either $J_\mu^1= \emptyset$ or $J_\mu^2=\emptyset$, and $A_{\bar K}(\cO_\mu)$ is generated by $\delta_i$, $1\leq i\leq s-1$.

The subset $\Pi_{\cO_\mu}=\Pi_\mu\subset\widehat{A_{ K}(\cO_\mu)}_{\kappa_0}$ is defined   as follows (via  $\widehat{A_K(\cO_\mu)}_{\kappa_0}\cong\widehat{A_{\bar K}(\cO_\mu)}$)
\beq
\label{char-bi}
\Pi_{\cO_\mu}=\{\chi\in \widehat{A_{\bar K}(\cO_\mu)}\mid \chi(\delta_r)=1 \ \text{if}\  r+1\notin \Omega_{\cO_\mu} \}\,.
\eeq
Note that this is well-defined since $i_a+1\in\Omega_{\cO_\mu}$, $a=1,\ldots,t$ when $\mu\in\Sigma_{b,1}$. In particular, we see that 
\beq\label{set of chars}
|\Pi_{\cO_\mu}|=\begin{cases}2^{l_{\mu}-1}\text{ (resp. $2^{l_\mu})$}&\text{ if $N$ is odd and $\mu\in\Sigma_{b,1}$ (resp. $\Sigma_{b,2}$)}\\
2^{l_{\mu}-2}\text{ (resp. $2^{l_\mu-1})$}&\text{ if $N$ is even and $\mu\in\Sigma_{b,1}$ (resp. $\Sigma_{b,2}$)}.
\end{cases}
\eeq

\subsection{Character sheaves}
To describe the character sheaves for type BDI, we first write down representations of the fundamental groups in~\eqref{explicit fundamental1} and~\eqref{explicit fundamental2}.

Recall the set $\Theta_{n,1}^{\kappa_0}$ (resp. $\Theta_{n,0}^{\kappa_0}$)  of irreducible representations of ${\widetilde{B}_{W_n}}$ (resp. ${\widetilde{B}_{W_n'}}$) defined in \S\ref{sec-split}. We can regard it as a set of simple $\bC[{}_0\widetilde{ B}_{W_n}]$-modules (resp. $\bC[{}_0\widetilde{ B}_{W_n'}]$-modules) via the natural projection ${\widetilde{B}_{W_n}}\to{}_0{\widetilde{B}_{W_n}}$ (resp. ${\widetilde{B}_{W_n'}}\to{}_0{\widetilde{B}_{W_n'}})$ where $\epsilon\mapsto 1$. 

 Let $\Theta_{n,1}^{\kappa_0,1}$ and $\Theta_{n,0}^{\kappa_0,1}$ denote the following sets of non-isomorphic simple $\bC[\widetilde B_{W_n}^1]$-modules
\ber\label{Theta1}
\Theta_{n,1}^{\kappa_0,1}&=&\{V_{\rho_1\otimes\rho_2, \tilde\chi_k}\mid\rho_1\in\on{Irr}\cH_{W_k,-1},\,\rho_2\in\on{Irr}\cH_{W_{n-k},-1},\, k\in[0,n]\},\\
\Theta^{\kappa_0,1}_{n,0}&=&\{ V_{\rho_1\otimes\rho_2, \tilde\chi_k}\mid\rho_1\in\on{Irr}\cH_{W_k,-1,1},\,\rho_2\in\on{Irr}\cH_{W_{n-k},-1,1},\, k\in[0,n]\},\nonumber
\eer
where $\tilde\chi_k$ is a set of representatives of $B_{W_n}$-orbits on $\widehat{(\bZ/2\bZ)^n}$ such that  $(W_n)_{\tilde\chi_k}\cong W_k\times W_{n-k}$.

Let $\Theta_{n,1}^{\kappa_0,2}$ and $\Theta_{n,0}^{\kappa_0,2}$ denote the following set of non-isomorphic simple  $\bC[{}_0\widetilde B_{W_n}]$-modules
\ber\label{Theta2}
&&\Theta_{n,1}^{\kappa_0,2}=\Theta_{n,1}^{\kappa_0}\text{ (see~\eqref{Theta0})},\\
&&\Theta^{\kappa_0,2}_{n,0}=\{V_{\rho_1\otimes\rho_2, \chi_k}\mid\rho_1\in\on{Irr}\cH_{W_k,-1,1},\,\rho_2\in\on{Irr}\cH_{W_{n-k},-1,1},\, k\in[0,n/2],\rho_1\not\cong\rho_2\},\nonumber \\
&&\cup\{V_{\rho\otimes\rho,\chi_{n/2}}^\delta\mid\rho\in\on{Irr}\cH_{W_{n/2},-1,1},\,\delta=\rm I,\rm{II}\}\text{ where $V_{\rho_1\otimes\rho_2, \chi_{n/2}}\cong V_{\rho_2\otimes\rho_1, \chi_n/2}$}\nonumber.
\eer

Let us write
\beqn
\varepsilon_t=1\text{ (resp. $0$) if $t$ is odd (resp. even)}.
\eeqn
Suppose that $\mu\in\Sigma_{b,1}$ (resp. $\Sigma_{b,2}$, $\emptyset$). For each $\psi\in\Theta_{m,\varepsilon_t}^{\kappa_0,1}$ (resp. $\Theta_{m,\varepsilon_t}^{\kappa_0,2}$, $\Theta_{m,0}^{\kappa_0}$), $\tau\in\cP(k)$, and $\phi\in\Pi_{\cO_\mu}$ (see~\eqref{char-bi}), let $L_\psi\boxtimes L_\tau\boxtimes\phi$ denote the representation  of $\pi_1^{ K}(\widecheck\cO_{m,k,\mu})$ such that the action of $\pi_1^{ K}(\widecheck\cO_{m,k,\mu}) $ factors through $\pi_1^{ \bar K}(\widecheck\cO_{m,k,\mu})$, where $\widetilde{B}_{W_m}^1$ (resp. ${}_0\widetilde{B}_{W_m}$, ${}_0\widetilde{B}_{W_m'}$) acts via $\psi$, $\widetilde{B}_{W_k}$ acts via the $B_{W_k}$-representation $L_\tau$, and $(\bZ/2\bZ)^{r_\mu}$ acts via $\phi$.   Let $\calT_{\psi,\tau,\phi}$ denote the corresponding  irreducible $ K$-equivariant local system on $\widecheck\cO_{m,k,\mu}$. We write $\cT_{\psi,\tau}$ (resp. $\cT_{\tau,\phi}$) instead of $\calT_{\psi,\tau,\phi}$ when $\mu=\emptyset$ (resp. $m=0$) etc. 

Recall the set $\Theta_{m,t}^{\kappa_1}$ of irreducible representations of $\widetilde{B}_{m,t}$ defined in \S\ref{sec-split} and \S\ref{sec-cuspn}.   For each $\sigma\in\cP_2(k)$ and $\rho\in\Theta_{m,t}^{\kappa_1}$, let $\cF_{\rho,\sigma}$ denote the irreducible $K$-equivariant local system on $\widecheck\cO_{1^m_+1^m_-2^k_+2^k_-\sqcup{\mu_t}}$ corresponding to the irreducible representation $L_\sigma\boxtimes\rho$ of $\widetilde B_{W_k}^1\times \widetilde B_{m,t}$ such that $\widetilde B_{W_k}^1$ acts via the $B_{W_K}$-representation $L_\sigma$, and $\widetilde B_{m,t}$ acts via $\rho$.  

 \begin{thm}\label{thm-type BD}
 Suppose that $( G,K)=(Spin_{N},K^{q+t,q})$.
Then
 \bern
&{\rm(i)}& \Char_{K}(\Lg_1)_{\kappa_0}=\left\{\on{IC}(\widecheck\cO_{1^m_+1^m_-2^k_+2^k_-\sqcup\mu},\cT_{\psi,\tau,\phi})\,|\,\mu\in \Sigma_{b,1},\,\psi\in\Theta_{m,\varepsilon_t}^{\kappa_0,1},\,\tau\in\cP(k),\, \phi\in\Pi_{\cO_\mu} \right\}\\
&&\cup\left\{\on{IC}(\widecheck\cO_{1^m_+1^m_-2^k_+2^k_-\sqcup\mu},\cT_{\psi,\tau,\phi})\,|\,m>0,\,\mu\in \Sigma_{b,2},\,\psi\in\Theta_{m,\varepsilon_t}^{\kappa_0,2},\,\tau\in\cP(k),\ \phi\in\Pi_{\cO_\mu} \right\}\\
&&\cup\left\{\on{IC}(\widecheck\cO_{2^k_+2^k_-\sqcup\mu}^\delta,\cT_{\tau,\phi})\,|\,\delta={\rm I,II},\,\mu\in \Sigma_{b,2},\,\tau\in\cP(k),\ \phi\in\Pi_{\cO_\mu} \right\}\\
&&\cup\left\{\on{IC}(\widecheck\cO_{1^m_+1^m_-2^k_+2^k_-},\cT_{\psi,\tau})\,|\,m>0,\,\rho\in\Theta_{m,0}^{\kappa_0},\tau\in\cP(k) \right\}\,(t=0)\\
&&\cup\left\{\on{IC}(\widecheck\cO_{2^{n/2}_+2^{n/2}_-}^\delta,\cT_{\tau})\,|\,\delta=\mathrm{\rm{I}},\mathrm{\rm{II}},{\rm III}, {\rm IV},\,\tau\in\cP(n/2)\right\}\,(t=0\text{ and $n$ even}).
\\
&{\rm (ii)}& \on{Char}_K(\Lg_1)_{\kappa_1}=\left\{\on{IC}(\widecheck\cO_{1^{m}_+1^m_-2^k_+2^k_-\sqcup{\mu_t}},\cF_{\rho,\sigma})\mid m>0,\sigma\in\cP_2(k),\,\rho\in\Theta_{m,t}^{\kappa_1}\right\}\\&&\cup\left\{\on{IC}(\widecheck\cO_{2^k_+2^k_-\sqcup\mu_t},\cF_{\rho,\sigma})\mid \sigma\in\cP_2(k),\,\rho\in\Theta_{0,t}^{\kappa_1}=\widehat{A_K(\cO_{\mu_t})}_{\kappa_1}\right\}\ (|t|\geq 2)\\&&\cup\left\{\on{IC}(\widecheck\cO^\delta_{2^k_+2^k_-\sqcup1_{sgn_t}},\cF_{\sigma})\mid \sigma\in\cP_2(k),\,\delta=\rm{I},\rm{II}\right\}\ (|t|=1)\\
&&\cup\left\{\on{IC}(\widecheck\cO^\delta_{2^k_+2^k_-},\cF_{\sigma})\mid \sigma\in\cP_2(k),\,\delta=\rm{I},\rm{II},\rm{III},\rm{IV}\right\}\ (t=0\text{ and $n$ even})\,.
\eern
\end{thm}
\begin{thm}\label{thm-type DIII}
 Suppose that $( G,K)$ is of type {\rm DIII}.
Then
 \bern
&& \Char_{K}(\Lg_1)_{\kappa_0}=\left\{\on{IC}(\widecheck\cO_{1^{2k}_+1^{2k}_-\sqcup\mu},\cL_{\tau})\,|\,\mu\in \Lambda_{b}^{n-2k,n-2k},\,\tau\in\cP(k) \right\}\\
&& \on{Char}_K(\Lg_1)_{\kappa_1}=\left\{\on{IC}(\Lg_1^{rs},\cL_\sigma\otimes\bC_{\chi_1})\mid \sigma\in\cP_2({n/2})\right\}.
\eern
\end{thm}
Here $\cL_\tau$ denotes the irreducible local system given by the irreducible representation $L_\tau$, where  $\pi_1^K(\widecheck\cO_{k,\mu})$ acts through  the action of $\pi_1^{\bar K}(\widecheck\cO_{k,\mu})=B_{W_k}$, and $\cL_\sigma\otimes\bC_{\chi_1}$ is defined in~\S\ref{nearby-d3}. 
Theorem~\ref{thm-type DIII} follows from~\cite[Theorem 6.2]{VX}, Corollary~\ref{corollary-Diii}, and~\eqref{Diiin}. We prove Theorem~\ref{thm-type BD} in the next section.

\begin{corollary}\label{coro-cuspidal}
 Theorem~\ref{theorem-cusp} from the introduction holds. 
Moreover, we have
\bern
&&\sum_n|\on{Char}^{\on{cusp}}_{K^{n\pm1,n}}(\Lg_1)_{\kappa_0}|x^n=\frac{1}{2}\prod_{s\geq 1}(1+x^{2s})^2(1+x^s)^2+\frac{3}{2}\prod_{s\geq 1}(1+x^{4s})(1+x^{2s})\\
&&\sum_{n \text{ odd}}|\on{Char}^{\on{cusp}}_{K^{n,n}}(\Lg_1)_{\kappa_0}|x^n=x\prod_{s\geq 1}(1+x^{4s})^4(1+x^{2s})^4\\
&&\sum_{n \text{ even}}|\on{Char}^{\on{cusp}}_{K^{n,n}}(\Lg_1)_{\kappa_0}|x^n=\frac{1}{4}\prod_{s\geq 1}(1+x^{4s-2})^4(1+x^{2s})^4+\frac{3}{2}\prod_{s\geq 1}(1+x^{4s-2})(1+x^{2s})
\\
&&|\on{Char}^{\on{cusp}}_{K^{m+\frac{t^2+t}{2},m+\frac{t^2-t}{2}}}(\Lg_1)_{\kappa_1}|=\text{coefficient of $x^m$ in }\eta_{m,t}\prod_{s\geq 1}(1+x^s) .
\eern
\end{corollary}

\begin{corollary}
\label{nil coro-2}
{\rm (i)} Assume that $(G,K)=(Spin_N,K^{p,q})$ and either $p$ or $q$ is even. We have
\beqn
\begin{gathered}
\on{Char}_{K}^{\mathrm{n}}(\Lg_1)_{\kappa_0}=\{\on{IC}(\cO_\mu,\cE_\phi)\,|\,\mu\in\Sigma_{b,1}^{p,q},\ \phi\in\Pi_{\cO_\mu}\}\cup\{\on{IC}(\cO_\mu^\delta,\cE_\phi)\,|\,\mu\in\Sigma_{b,2}^{p,q},\ \phi\in\Pi_{\cO_\mu},\,\delta={\rm I,II}\},
\end{gathered}
 \eeqn
where $\Pi_{\cO_\mu}$ is defined in~\eqref{char-bi}. Moreover, $|\on{Char}_{K^{q+t,q}}^{\mathrm{n}}(\Lg_1)_{\kappa_0}|=\text{coefficient of $x^q$ in }$
\bern
\begin{cases}\displaystyle{\frac{1}{2(1+x^{t})}\prod_{s\geq 1}\frac{(1+x^{2s-1})^2}{(1-x^{2s})^2}+\frac{3(1+x^{t})}{2(1+x^{2t})}\prod_{s\geq 1}\frac{1+x^{4s-2}}{(1-x^{2s})^2}}&\text{ if $t$ is odd}\\
\displaystyle{\frac{1}{2(1+x^{t})}\prod_{s\geq 1}\frac{(1+x^{2s})^2}{(1-x^{2s})^2}+\frac{3(1+x^{t})}{2(1+x^{2t})}\prod_{s\geq 1}\frac{1+x^{4s}}{(1-x^{2s})^2}}&\text{ if $t,q$ both even}\,.\end{cases}
\eern 
{\rm (ii)} We have
\beqn
\on{Char}_{K^{\frac{t^2+t}{2},\frac{t^2-t}{2}}}^{\mathrm{n}}(\Lg_1)_{\kappa_1}=\left\{\IC(\cO_{\mu_t},\cE_{\phi})\mid\phi\in\widehat{A_K(\cO_{\mu_t})}_{\kappa_1}\right\}.
\eeqn
Moreover, $|\on{Char}_{K^{\frac{t^2+t}{2},\frac{t^2-t}{2}}}^{\mathrm{n}}(\Lg_1)_{\kappa_1}|=\eta_{0,t}$.
\end{corollary}

\section{Proof of the main theorem~\ref{thm-type BD} and corollaries}\label{sec-proof}

In this section we assume that $(G,K)$ is of type BDI unless otherwise stated. We prove  Theorem~\ref{thm-type BD} and its corollaries. To prove Theorem~\ref{thm-type BD},  as in~\cite{VX}, we show that the sheaves in the theorem are indeed character sheaves by constructing them using parabolic induction. We then conclude by showing that the number of the  sheaves we have constructed coincides with the number of  character sheaves.

\subsection{Parabolic induction} 
In this subsection we study parabolic inductions of character sheaves in $\on{Char}_{(L^\theta)^0}(\Ll_1)$ for a family of $\theta$-stable Levi subgroups $L$ contained in $\theta$-stable parabolic subgroups.

We first consider sheaves in $\Char_{K}(\Lg_1)_{\kappa_0}$, identified with $\Char_{\bar K}(\Lg_1)$. We will work in the setting of $(\bar G,\bar K)=(SO_N, SO_p\times SO_q)$.  We follow~\cite[\S7.2, \S7.3]{VX}. Recall  the family of $\theta$-stable parabolic subgroups $P_{m,k,\mu}$ together with their $\theta$-stable Levi subgroups $L_{m,k,\mu}$ such that
$
\bar K.(\Lp_{m,k,\mu})_1=\overline{\widecheck\cO}_{m,k,\mu},\ \cO_{m,k,\mu}\in\underline{\cN_1^{\text{cs},\kappa_0}}
$
,
$L_{m,k,\mu}\cong SO_{2m+\varepsilon_N}\times GL_{2k}\times GL_1^{n-m-2k},$ and  
$(L_{m,k,\mu}^\theta)^0\cong SO_{m+\varepsilon_N}\times SO_{m}\times GL_k\times GL_k\times GL_1^{n-m-2k}$, where $\varepsilon_N=1$ (resp. $0$) if $N$ is odd (resp. even). Note that when $m=0$ and $\mu\in\Sigma_2$ (resp. $\mu=\emptyset$), the $\tilde{K}$-conjugacy class of $P_{0,k,\mu}$ in $G/P$ decomposes into $2$ (resp. $4$) $\bar K$-conjugacy classes.

Assume that either $m\neq 0$ or $k\neq 0$.
In what follows we write $ L=L_{m,k,\mu}$, $P=P_{m,k,\mu}$, and $\overline{\widecheck\cO}=\bar K.\Lp_1$ etc. We describe the  equivariant fundamental groups $\pi_1^{P_{\bar K}}(\Lp_1^r)$, $\pi_1^{\bar K}(\widecheck\cO)$ and $\pi_1^{(L^\theta)^0}(\Ll_1^{rs})$, where 
$
P_{\bar K}=P\cap \bar K,\ \  \Lp_1^r=\Lp_1\cap{\widecheck\cO},
$
and $\Ll_1^{rs}$ is the set of regular semisimple elements of $\Ll_1$ with respect to the symmetric pair $(L,(L^\theta)^0)$.
We have
 \begin{eqnarray*}
  \pi_1^{(L^\theta)^0}(\Ll_1^{rs})\cong {}_0\widetilde{B}_{W_m}\text{ (resp. ${}_0\widetilde{B}_{W_m'}$)}\times B_{W_k}, && \text{$N$ odd (resp. even)}.
 \end{eqnarray*}
Consider the  the map $\pi:\bar K\times ^{P_{\bar K}}\Lp_1\to\overline{\widecheck\cO}$.
 Let $x\in\widecheck\cO$.  Using similar argument as in ~\cite[Lemma 7.1]{VX}, one can show that $\pi^{-1}(x)$ has $2^{l_\mu}$ irreducible components, and the $A_K(x)$-action on the set of irreducible components is transitive if $N$ is odd, and has two orbits if $N$ is even. Moreover, let $s_\mu=r_\mu-l_\mu+1$ (resp. $r_\mu-l_\mu$) if $N$ is odd and $\mu\in\Sigma_1$ (resp. $\Sigma_2$), and $s_\mu=r_\mu-l_\mu+2$ (resp. $r_\mu-l_\mu+1$) if $N$ is even and $\mu\in\Sigma_1$ (resp. $\Sigma_2$). Then
\begin{subequations}
\begin{eqnarray}\label{equivpi1-p1}
&\pi_1^{P_{\bar{K}}}(\Lp_1^r)\cong\pi_1^{\bar K}(\widecheck\cO)\cong{}_0\widetilde{B}_{W_m'}\times \widetilde{B}_{W_k}^1&\text{ if $\mu=\emptyset$}
\\
\label{equivpi1-p2}
&\pi_1^{P_{\bar{K}}}(\Lp_1^r)\cong {}_0\widetilde{B}_{W_m}\text{ (resp. ${}_0\widetilde{B}^0_{W_m}$)}\times\widetilde{B}_{W_k}^1\times (\bZ/2\bZ)^{s_\mu}&\text{$N$ odd (resp. even)}
\end{eqnarray}
\end{subequations}
where ${}_0\widetilde{B}_{W_m}^0=(\bZ/2\bZ)^{m-1}\rtimes B_{W_m}^0$ and $B_{W_m}^0$ is the subgroup of $B_{W_m}$ defined  as 
the kernel of the map $B_{W_m}\to\bZ/2\bZ$ sending  a word in the braid generators $\sigma_1,\ldots,\sigma_m$ ($\sigma_m$ is the special one) to 0 (resp. 1) if the braid generator $\sigma_m$ appears even (resp. odd) number of times.
Furthermore, the factor  $(\bZ/2\bZ)^{s_\mu-l_\mu}$  is given by $\{a\in A_{\bar K'}(\cO_\mu)\,|\,\phi(a)=1\text{ for all }\phi\in\Pi_{\cO_\mu}\}$.

We consider the parabolic induction of full support character sheaves on $\Ll_1$. For each irreducible representation $\psi$ of $\pi_1^{(L^\theta)^0}(\Ll_1^{rs})$, we write $\psi$ also for the irreducible representation of $\pi_1^{P_{\bar K}}(\Lp_1^r)$ obtained via pull-back from the surjective map $\Phi:\pi_1^{P_{\bar K}}(\Lp_1^r)\to \pi_1^{(L^\theta)^0}(\Ll_1^{rs})$. Note that when $N$ even and $\mu\neq\emptyset$, the map $\Phi$  gives rise to a surjective map
$\zeta:{}_0\widetilde{B}^0_{W_m}\twoheadrightarrow{}_0\widetilde{B}_{W_m'}.$ 
As in~\cite{VX}, one can check that  $\Theta_{n,|t|}^{\kappa_0,1}$ (resp. $\Theta_{n,
|t|}^{\kappa_0,2}$) coincides with the set of non-isomorphic simple $\bC[\widetilde B_{W_n}^1]$-modules (resp. $\bC[{}_0\widetilde B_{W_n}]$-modules) that arise as a direct summand of $\bC[{\widetilde{B}^1_{W_n}}]\otimes_{\bC[_0{\widetilde{B}_{W_n}}]}\psi$ (resp. $\bC[{}_0{\widetilde{B}_{W_n}}]\otimes_{\bC[_0{\widetilde{B}_{W_n}}^0]}\psi)$, $\psi\in\Theta_{n,|t|}^{\kappa_0}$, $|t|\leq 1$.
 The following claims then follow from the description of $\pi_1^{\bar K}(\widecheck\cO)$ in~\eqref{explicit fundamental1}, the description of $\pi_1^{P_{\bar K}}(\Lp_1^r)$ in~\eqref{equivpi1-p1}-\eqref{equivpi1-p2}.

 Let $\cL_\rho\boxtimes\cL_\tau$, $\rho\in\Theta_{m,1}^{\kappa_0}$ (resp. $\Theta_{m,0}^{\kappa_0}$), and $\tau\in\cP(k)$, be the  irreducible $(L^\theta)^0$-equivairant local system on $\Ll_1^{rs}$ given by the irreducible representation $\rho\boxtimes L_\tau$ of $\pi_1^{(L^\theta)^0}(\Ll_1^{rs})\cong {}_0\widetilde{B}_{W_m}\times B_{W_k}$ \text{(resp. ${}_0\widetilde{B}_{W_m'}\times B_{W_k}$)}, where ${}_0\widetilde{B}_{W_m}$ (resp. ${}_0\widetilde{B}_{W_m'}$) acts via $\rho$ and $B_{W_k}$ acts via $L_\tau$.  Then
 \bern
&&\bigoplus_{\phi\in\Pi_{\cO_\mu}}\on{IC}(\widecheck\cO,\oplus_{\psi\in\Theta_{m,1}^{\kappa_0,i}}\cT_{\psi,\tau,\phi})[-]\ \dsum\  \on{Ind}_{\Ll_1\subset\fp_1}^{\Lg_1}\on{IC}(\Ll_1^{rs},\oplus_{\rho\in\Theta_{m,1}^{\kappa_0}}\cL_\rho\boxtimes\cL_\tau),\,\mu\in\Sigma_i,i=1,2,
 \\
&&\on{IC}(\widecheck\cO_{m,k,\emptyset},\cT_{\rho,\tau})[-]\ \dsum\  \on{Ind}_{\Ll_1\subset\fp_1}^{\Lg_1}\on{IC}(\Ll_1^{rs},\cL_\rho\boxtimes\cL_\tau),\text{ when $m>0$},\\
&&\on{IC}(\widecheck\cO_{0,n/2,\emptyset}^\delta,\cT_{\tau})[-]\ \dsum\ \on{Ind}_{\Ll_1^\delta\subset\fp_1^\delta}^{\Lg_1}\on{IC}((\Ll_1^\delta)^{rs},\cL_\tau), \ \delta=\mathrm{\rm{I},\rm{II,III,IV}}.
 \eern
 
We consider next the sheaves in $\Char_K(\Lg_1)_{\kappa_1}$. Let $P$ be a $\theta$-stable subgroup such that $\pi(P)$ stabilises the flag $0\subset V_{2k}\subset V_{2k}^\p\subset V$, where $V_{2k}=\on{span}\{e_i,f_i,i=1,\ldots,k\}$. Let $L\subset P$ be the $\theta$-stable Levi subgroup such that $\pi(L)\cong GL(V_{2k})\oplus SO(V_{2k}^\p/V_{2k})$. We have
\beqn
\pi(L^\theta)^0\cong GL_k\times GL_k\times SO_{p-2k}\times SO_{q-2k}.
\eeqn
Consider the stratum $\widecheck\cO^{\Ll_1}:=\widecheck\cO^{\Ll_1}_{1^k_+1^k_-\boxtimes 1^m_+1^m_-\sqcup{\mu_t}}$ in $\Ll_1$. We have
\beqn
\pi_1^{L^\theta}(\widecheck\cO^{\Ll_1})\cong B_{W_k}\times\bZ/2\bZ\times \widetilde B_{m,t}.
\eeqn
For each $\sigma\in\cP_2(k)$ and $\rho\in\Theta_{m,t}^{\kappa_1}$, let $\cL_{\rho,\sigma}$ denote the local system corresponding to the $\pi_1^{L^\theta}(\widecheck\cO^{\Ll_1})$-representation where $B_{W_k}$ acts via $L_\sigma$ of $W_k$, $\bZ/2\bZ$ acts via the non-trivial character and $\widetilde B_{m,t}$ acts via $\rho$. Consider the IC sheaf $\IC(\widecheck\cO^{\Ll_1},\cL_{\rho,\sigma})$. This is a character sheave on $\Ll_1$. Note that the image of $K\times^{P_K}(\overline{\widecheck\cO^{\Ll_1}}+(\Ln_P)_1)$ under the map $\pi:K\times^{P_K}\Lp_1\to K.\Lp_1$ is $\overline{\widecheck\cO_{1^m_+1^m_-2^k_+2^k_-\sqcup\mu_t}}$. The same argument as before shows that
\beqn
\IC(\widecheck\cO_{1^m_+1^m_-2^k_+2^k_-\sqcup\mu_t},\cF_{\rho,\sigma})\ \dsum\  \on{Ind}_{\Ll_1\subset\Lp_1}^{\Lg_1}\IC(\widecheck\cO^{\Ll_1},\cL_{\rho,\sigma}).
\eeqn

Since Fourier transform commutes with parabolic induction, we conclude that  the IC sheaves in Theorem~\ref{thm-type BD} are character sheaves. Since they are pairwise non-isomorphic,  it remains to check that the number of the IC sheaves equals the number of character sheaves. This will be done in the next subsections.

\subsection{The number of character sheaves}

In this subsection, we determine the numbers  $|\on{Char}_K(\Lg_1)_{\kappa_0}|$ and $|\on{Char}_K(\Lg_1)_{\kappa_1}|$ of character sheaves. 
\begin{proposition} \label{number1}Suppose that $( G, K)=(Spin_N,K^{q+t,q})$. We have that
\ber
\label{numbertrivial}&&|\on{Char}_K(\Lg_1)_{\kappa_0}|=\text{ coefficient of $x^q$ in}\\
&&\qquad\qquad\frac{1}{2(1+x^{t})}\prod_{s\geq 1}\frac{1+x^s}{(1-x^s)^3}+\frac{3(1+x^{t})}{2(1+x^{2t})}\prod_{s\geq 1}\frac{(1+x^{2s})^2}{(1-x^{2s})^3}+\frac{9}{4}\delta_{t,0}\prod_{s\geq1}\frac{1}{1-x^{2s}}\,\nonumber
\\
\label{numberntrivial}&&\text{$|\on{Char}_K(\Lg_1)_{\kappa_1}|=$  coefficient of $x^{N-t^2}$ in }\ \eta_{\frac{N-t^2}{2},t}\prod_{s\geq 1}\frac{1}{(1-x^{4s})(1-x^{2s})}.
\eer
\end{proposition}

\begin{lemma}\label{lemma-n1}We have
\beqn
\sum_{\lambda\in\Sigma_2^{q+t,q}\cup\Sigma_3^{q+t,q}}2^{r_\lambda}=\text{ coefficient of $x^q$ in }\frac{1+x^{t}}{1+x^{2t}}\prod_{s\geq 1}\frac{(1+x^{2s})^2}{(1-x^{2s})^3}.
\eeqn

\end{lemma}
\begin{proof}
Mimicking the proof of~\cite[Proposition C.2]{VX}, we see that the 2-variable generating function of $\sum_{\lambda\in\Sigma_2\cup\Sigma_3}2^{r_\lambda}$ is
\beq\label{eqn1-pf}
\begin{gathered}
\prod_{m\geq 0}\frac{1+u^{2m+2}v^{2m+1}}{1-u^{2m+2}v^{2m+1}}\prod_{m\geq 0}\frac{1+u^{2m}v^{2m+1}}{1-u^{2m}v^{2m+1}}\prod_{m\geq 1}\frac{1}{1-u^{2m}v^{2m}}\\
+\prod_{m\geq 0}\frac{1+u^{2m+1}v^{2m+2}}{1-u^{2m+1}v^{2m+2}}\prod_{m\geq 0}\frac{1+u^{2m+1}v^{2m}}{1-u^{2m+1}v^{2m}}\prod_{m\geq 1}\frac{1}{1-u^{2m}v^{2m}}.
\end{gathered}
\eeq
The first term of~\eqref{eqn1-pf} equals
\beqn
\frac{1}{(u^2v^2;u^2v^2)_\infty}\sum_{n}\frac{(-1;u^2v^2)_n}{(u^2v^2;u^2v^2)_n}(u^2v)^n\sum_{n}\frac{(-1;u^2v^2)_n}{(u^2v^2;u^2v^2)_n}v^n.
\eeqn
Picking out the terms of the form $u^{m+t}v^m$ we get
\beqn
\frac{1}{(u^2v^2;u^2v^2)_\infty}\sum_{n}\frac{(-1;u^2v^2)_{n+t}}{(u^2v^2;u^2v^2)_n}\frac{(-1;u^2v^2)_{n+t}}{(u^2v^2;u^2v^2)_n}u^{2n+2t}v^{2n+t}.
\eeqn
Setting $u=v=x$, we get
$
\displaystyle{\frac{2x^{3t}}{1+x^{4t}}\frac{(-x^4;x^4)_\infty^2}{(x^4;x^4)_\infty^3}}.
$
Similarly the second term of~\eqref{eqn1-pf} gives us 
$
\displaystyle{\frac{2x^{t}}{1+x^{4t}}\frac{(-x^4;x^4)_\infty^2}{(x^4;x^4)_\infty^3}}.
$
Thus we have
\beqn
\sum_{q}\left(2\sum_{\lambda\in\Sigma_2^{q+t,q}\cup\Sigma_3^{q+t,q}}2^{r_\lambda}\right)x^{2q+t}=\frac{2(x^{3t}+x^t)}{1+x^{4t}}\frac{(-x^4;x^4)_\infty^2}{(x^4;x^4)_\infty^3}
\eeqn
and the lemma follows.
\end{proof}

\begin{proof}[Proof of Proposition~\ref{number1}]Note that
$
|\on{Char}_K(\Lg_1)_{\kappa_i}|=|\cA_K(\Lg_1)_{\kappa_i}|$, $i=0,1$, $|\cA_K(\Lg_1)_{\kappa_0}|=|\cA_{\bar K}(\Lg_1)|$ and
 that
$
|A_{\bar K}(x_\lambda)|=|A_{\widetilde K}(x_\lambda)|/2$, if $\lambda\in\Sigma_1$, $ 
|A_{\bar K}(x_\lambda)|=|A_{\widetilde K}(x_\lambda)|$,  if $\lambda\notin\Sigma_1.
$
Thus
\bern
&&|\Char_{\bar K}(\Lg_1)|=\sum_{\lambda\in\Sigma_1^{q+t,q}}2^{r_\lambda}+2\sum_{\lambda\in\Sigma_2^{q+t,q}}2^{r_\lambda}+4\sum_{\lambda\in\Sigma_3^{q+t,q}}2^{r_\lambda}\\
&&|\Char_{\widetilde K}(\Lg_1)|=\sum_{\lambda\in\Sigma_1^{q+t,q}}2^{r_\lambda+1}+\sum_{\lambda\in\Sigma_2^{q+t,q}}2^{r_\lambda}+2\sum_{\lambda\in\Sigma_3^{q+t,q}}2^{r_\lambda}.
\eern
It follows that
\beqn 
|\Char_{\bar K}(\Lg_1)|=\frac{1}{2}|\Char_{\widetilde K}(\Lg_1)|+\frac{3}{2}\sum_{\lambda\in\Sigma_2^{q+t,q}\cup\Sigma_3^{q+t,q}}2^{r_\lambda}+\frac{3}{2}\mathbf{p}\left(\frac{N}{4}\right)\delta_{t,0}.
\eeqn
Hence~\eqref{numbertrivial} of the proposition follows from Lemma~\ref{lemma-n1} and the following equation from~\cite[Corollary 8.6]{VX} 
\beqn
|\Char_{\widetilde K}(\Lg_1)|=\text{ coefficient of $x^q$ in }\frac{1}{1+x^{t}}\prod_{s\geq 1}\frac{1+x^s}{(1-x^s)^3}+\frac{3\delta_{t,0}}{2}\prod_{s\geq 1}\frac{1}{1-x^{2s}}.
\eeqn

To prove~\eqref{numberntrivial}, following~\cite[\S 14]{L1}, we write $P_{N,t}$ for the set of partitions of $N$ into distinct odd parts (and no even parts) such that the number of parts of the form $4m+1$ minus the number of parts of the form $4m+3$ equals $t$. Then we have
$
|P_{N,t}|=\mathbf{p}\left(\frac{N-(2t^2-t)}{4}\right).
$
 Let $\lambda$ be a signed Young diagram such that $p_i\leq 1,q_i\leq 1$ for all odd $\lambda_i$. Recall the sets $J_1,J_2$ defined in~\eqref{set 1} and~\eqref{set 2}. Then $\lambda_i$, $i\in J_1$ form a partition in $P_{N_1,t_1}$ for some $N_1\leq N,t_1\in\bZ$ and $\lambda_i$, $i\in J_2$ form a partition in $P_{N_2,t_2}$ for some $N_2\leq N,t_2\in\bZ$. Moreover, $t_1-t_2=t$. Thus it follows from Lemma~\ref{lemma-component gp} that
\bern
|\cA_K(\Lg_1)_{\kappa_1}|&=&\eta_{\frac{N-t^2}{2},t}\sum_{\substack{N_1+N_2\leq N\\t_1\in\bZ,t_2\in\bZ\\t_1-t_2=t}}\mathbf{p}\left(\frac{N_1-(2t_1^2-t_1)}{4}\right)\mathbf{p}\left(\frac{N_2-(2t_2^2-t_2)}{4}\right)\mathbf{p}\left(\frac{N-N_1-N_2}{4}\right)\\
&=&\text{coefficient of $x^{N-t^2}$ in }\eta_{\frac{N-t^2}{2},t}\prod_{s\geq 1}\frac{1}{(1-x^{4s})(1-x^{2s})}\,.
\eern
Here we have used the Jacobi triple product formula to deduce  that
$
\sum_{t_1-t_2=t}x^{2t_1^2-t_1+2t_2^2-t_2}=\sum_{t_2}x^{(2t_2+t)^2-(2t_2+t)+t^2}=x^{t^2}\prod_{s\geq 1}(1-x^{4s})(1+x^{2s}).
$
\end{proof}

Let
\beq\label{number-t}
T_{p,q}^0=|\on{Char}_{ K^{p,q}}(\Lg_1)_{\kappa_0}|\,\text{ and }  T_N^0=\sum_{p+q=N} T_{p,q}^0.
\eeq
\begin{corollary}\label{numbert} We have
\bern
&&\sum_{N=0}^\infty T_N^0x^N=\frac{1}{4}\prod_{s\geq 1}\frac{(1+x^{2s-1})^2}{(1-x^{4s})(1-x^{2s-1})^2}+\frac{3}{2}\prod_{s\geq 1}\frac{1+x^{2s-1}}{(1-x^{4s})(1-x^{2s-1})}+\frac{9}{4}\prod_{s\geq 1}\frac{1}{1-x^{4s}}\\
&&\sum_{n=0}^\infty T_{2n+1}^0x^{n}=\prod_{s\geq 1}\frac{(1+x^{2s})^4(1+x^{s})^3}{(1-x^{s})}+3\prod_{s\geq 1}\frac{(1+x^{4s})^2(1+x^{2s})(1+x^{s})}{(1-x^{s})}\\
&&\sum_{n=0}^\infty  T_{2n}^0x^{n}=\frac{1}{4}\prod_{s\geq 1}\frac{(1+x^{2s-1})^4(1+x^{s})^3}{(1-x^{s})}+\frac{3}{2}\prod_{s\geq 1}\frac{(1+x^{4s-2})^2(1+x^{2s})(1+x^{s})}{(1-x^{s})}\\&&\hspace{1in}+\frac{9}{4}\prod_{s\geq 1}\frac{1}{1-x^{2s}}.
\eern
\end{corollary}
\begin{proof}
Using the Ramanujan's ${}_1\psi_1$ formula as in~\cite[Appendix B]{CVX}, we see that
\bern
&&\sum_{k=-\infty}^\infty\frac{x^k}{1+x^{2k}}=\frac{1}{2}\prod_{s\geq 1}\frac{(1+x^{2s-1})^2(1-x^{2s})^2}{(1-x^{2s-1})^2(1+x^{2s})^2}
\\
&&\sum_{k=-\infty}^\infty\frac{x^k(1+x^{2k})}{1+x^{4k}}={}_1\psi_1(-1,-x^4;x^4;x)=\prod_{s\geq 1}\frac{(1+x^{2s-1})(1-x^{4s})^2}{(1-x^{2s-1})(1+x^{4s})^2}.
\eern
Moreover, it follows from
\bern
&&\sum_{k=-\infty,\,k\text{ odd}}^\infty\frac{x^k(1+x^{2k})}{1+x^{4k}}=\frac{2x^3}{1+x^4}{}_1\psi_1(-x^{-4},-x^4;x^8;x^2)=2x\prod_{s\geq 1}\frac{(1+x^{4s-2})(1-x^{8s})^2}{(1-x^{4s-2})(1+x^{8s-4})^2}\\&&\sum_{k=-\infty,\,k\text{ even}}^\infty\frac{x^k(1+x^{2k})}{1+x^{4k}}={}_1\psi_1(-1,-x^8;x^8;x^2)=\prod_{s\geq 1}\frac{(1+x^{4s-2})(1-x^{8s})^2}{(1-x^{4s-2})(1+x^{8s})^2}
\eern
that
\bern
\prod_{s\geq 1}\frac{1+x^{2s-1}}{1-x^{2s-1}}&=&2x\prod_{s\geq 1}(1+x^{8s})^2(1+x^{4s})(1+x^{2s})^2+\prod_{s\geq 1}(1+x^{8s-4})^2(1+x^{4s})(1+x^{2s})^2.
\eern
Recall that (~\cite[Appendix C.3]{VX})
\beq\label{eqn-oeterms}
\prod_{s\geq 1}\frac{(1+x^{2s-1})^2}{(1-x^{2s-1})^2}=4x\prod_{s\geq 1}(1+x^{4s})^4(1+x^{2s})^4+\prod_{s\geq 1}(1+x^{4s-2})^4(1+x^{2s})^4.
\eeq
The corollary follows from Proposition~\ref{number1}.
\end{proof}

\subsection{Comparing the number of sheaves} \label{ssec-number}The number of sheaves in  Theorem~\ref{thm-type BD} (ii) equals the coefficient of $x^{{N-t^2}}$ in
\beqn
\eta_{\frac{N-t^2}{2},t}\prod_{s\geq 1}\frac{1}{(1-x^{4s})^2}\prod_{s\geq 1}(1+x^{2s})=\eta_{\frac{N-t^2}{2},t}\prod_{s\geq 1}\frac{1}{(1-x^{4s})(1-x^{2s})}\stackrel{\eqref{numberntrivial}}=|\Char_K(\Lg_1)_{\kappa_1}|
\eeqn
since $\eta_{\frac{N-t^2}{2},t}=\eta_{\frac{N-4k-t^2}{2},t}$ (see~\eqref{eqn-mpq}). Thus part (ii) of Theorem~\ref{thm-type BD} follows.

Let us write $T_{p,q}'$ for the number of sheaves in Theorem~\ref{thm-type BD} (i) and let $T_N'=\sum_{p+q=N}T'_{p,q}$. It suffices to show that $T_N'=T_N^0$ (see~\eqref{number-t}). Since $0\leq T'_{N}\leq T_{N}^0$ by definition, it will then follow that $T_{p,q}'= T^0_{p,q}.$

Let $f_{m,t}^i=|\Theta_{m,t}^{\kappa_0,i}|$, $i=1,2$, $t=0,1$, and $f_{m,0}=|\Theta_{m,0}^{\kappa_0}|$, see~\eqref{Theta0},~\eqref{Theta1} and~\eqref{Theta2}. We have
\ber
\label{T1}&&T'_{p,q}=\sum_{k}\sum_{m\geq1}\sum_{\mu\in\Sigma^{p-2k-m,q-2k-m}_{b,1}}\mathbf{p}(k)f_{m,1}^12^{l_\mu-1}+\sum_{k}\sum_{m\geq 1}\sum_{\mu\in\Sigma^{p-2k-m,q-2k-m}_{b,2}}\mathbf{p}(k)f_{m,1}^22^{l_\mu}\\
&&\qquad+\sum_{k}\sum_{\mu\in\Sigma^{p-2k,q-2k}_{b,1}}\mathbf{p}(k)2^{l_\mu-1}+2\sum_{k}\sum_{\mu\in\Sigma^{p-2k,q-2k}_{b,2}}\mathbf{p}(k)2^{l_\mu}\qquad \text{if $p+q$ is odd}\nonumber,
\\
&&T_{p,q}'=\sum_{k}\sum_{m\geq1}\sum_{\mu\in\Sigma^{p-2k-m,q-2k-m}_{b,1}}\mathbf{p}(k)f_{m,0}^12^{l_\mu-2}+\sum_{k}\sum_{m\geq 1}\sum_{\mu\in\Sigma^{p-2k-m,q-2k-m}_{b,2}}\mathbf{p}(k)f_{m,0}^22^{l_\mu-1}\nonumber\\
&&\qquad+\sum_{k}\sum_{\mu\in\Sigma^{p-2k,q-2k}_{b,1}}\mathbf{p}(k)2^{l_\mu-2}+2\sum_{k}\sum_{\mu\in\Sigma^{p-2k,q-2k}_{b,2}}\mathbf{p}(k)2^{l_\mu-1}\nonumber\\
&&\qquad+\delta_{p,q}\sum_k\sum_{m\geq 1}\mathbf{p}(k)f_{m,0}+4\delta_{p,q}\mathbf{p}(n/2)\text{\hspace{1.5in} if $p+q$ is even}.\nonumber
\eer
Let
\beq\label{tbpq}
\begin{gathered}
{b}_{p,q}^i=\sum_{\mu\in\Sigma_{b,i}^{p,q}}|\Pi_{\cO_\mu}|,\,b^i_N=\sum_{p+q=N}b_{p,q}^i,\,\,i=1,2,\ 
\tilde{b}_{p,q}=2b_{p,q}^1+b_{p,q}^2, \tilde b_{N}=\sum_{p+q=N}\tilde{b}_{p,q}.
\end{gathered}
\eeq
Then $
\tilde{b}_{p,q}=|\on{Char}^\rn_{\widetilde K^{p,q}}(\Lg_1)|,\ \tilde b_{2n+\varepsilon}=\sum_{p+q=2n+\varepsilon}\sum_{\mu\in\Sigma_{b}^{p,q}}2^{l_\mu-1+\varepsilon}
$, $\varepsilon=0,1$. Moreover, by~\cite{VX} we have
\ber
\label{bb-odd}&&\sum_{n=0}^\infty \tilde b_{2n+1}x^{n}=2\prod_{s\geq 1}{(1+x^{s})^2}{(1+x^{2s})^2},\ 
\sum_{n=0}^\infty \tilde b_{2n}x^{n}=\frac{1}{2}\prod_{s\geq 1}{(1+x^{s})^2}{(1+x^{2s-1})^2}.
\eer
\begin{lemma}
We have
\ber
\label{eqn-bio}&&\sum_{p=0}^{\infty}b_{2n+1}^2x^{n}=2\prod_{s\geq 1}(1+x^{s})^2(1+x^{4s}),\,\,\sum_{p=0}^{\infty}b_{2n}^2x^{n}=\prod_{s\geq 1}(1+x^{s})^2(1+x^{4s-2}).
\eer
\end{lemma}
\begin{proof}We follow the proof in~\cite[Appendix B]{CVX}. Let $\cP^{odd}(N)$ denote the set of partitions of $N$ into odd parts. We write a partition $\lambda\in\cP^{odd}(N)$ as
$\lambda=(2\mu_1+1)+(2\mu_2+1)+\cdots+(2\mu_{s}+1)
$
where $\mu_1\geq\mu_2\geq\cdots\geq\mu_{s}\geq 0$. Note that $s\equiv N\mod 2$. We set
$
wt_\lambda=2^{\#\{1\leq j\leq  (s-1)/2\,|\,\mu_{2j-1}\geq\mu_{2j}+2\}}\text{ when $N$ is odd};
$
$
wt_\lambda=2^{\#\{1\leq j\leq s/2-1\,|\,\mu_{2j}\geq\mu_{2j+1}+2\}}\text{ when $N$ is even}.
$
We have
\beqn
b_{N}^2=\sum_{\lambda\in\cP^{odd}(N)}wt_\lambda.
\eeqn
The first equation is equivalent to
$
\sum_{n=0}^\infty b_{2n+1}^2 x^{2n+1}=x (-x^8;x^8)_\infty (-x^2;x^2)_\infty^2.
$
This can be seen as follows. Suppose $\lambda$ with odd parts has an odd number of parts, say $2k+1$. Consider the columns of $\lambda$, which have possible sizes $1,2,\cdots, 2k+1.$ The part $2k+1$ occurs an odd number of times, the generating function is
$
\displaystyle{\frac{x^{2k+1}}{1-x^{4k+2}}}.
$
 The part $2k$ occurs an even number of times, the generating function is
$
\displaystyle{\frac{1}{1-x^{4k}}}.
$
 The part $2k-1$ occurs an even number of times, the weighted generating function is
$
\displaystyle{1+x^{4k-2}+\frac{2x^{8k-4}}{1-x^{4k-2}}}$$=\displaystyle{\frac{1+x^{8k-4}}{1-x^{4k-2}}}.
$
This continues down to part size 1, to obtain the generating function
\bern
\sum_{n=0}^\infty b_{2n+1}^2 x^{2n+1} &=&
\sum_{k=0}^\infty \frac{x^{2k+1}}{1-x^{4k+2}}
\frac{1}{\prod_{j=1}^k(1-x^{4j})}
\prod_{i=1,{\text{odd}}}^{2k-1}\frac{1+x^{4i}}{1-x^{2i}}\\&=&\frac{x}{1-x^2}\sum_{k=0}^{\infty}\frac{(-\mathbf{i}x^2;x^4)_k(\mathbf{i}x^2;x^4)_k}{(x^6;x^4)_k(x^4;x^4)_k}x^{2k}=\frac{x}{1-x^2} \frac{(\mathbf{i}x^4;x^4)_\infty (-\mathbf{i}x^4;x^4)_\infty}{(x^6;x^4)_\infty (x^2;x^4)_\infty}\\&=&x (-x^8;x^8)_\infty (-x^2;x^2)_\infty^2,
\eern
where in the second last equality we have applied the $q$-Gauss theorem~\cite{CVX} with $q\rightarrow x^4$ and $a=-\mathbf{i}x^2,\,b=\mathbf{i}x^2,\,c=x^6$. The second equation is equivalent to
$
\sum_{n=0}^\infty b_{2n}^2 x^{2n}=\frac{1}{2}(-x^4;x^8)_\infty (-x^2;x^2)^2_\infty.
$
Assume $\lambda$ is a partition into odd parts with $2k$ parts. We argue as 
above, this time the even parts have weights.
Thus
\begin{eqnarray*}
\sum_{n=0}^\infty b_{2n}^2 x^{2n} &=&
\sum_{k=0}^\infty \frac{x^{2k}}{1-x^{4k}}
\frac{1}{\prod_{j=1}^k(1-x^{4j-2})}
\prod_{i=1,{\text{even}}}^{2k-2} \frac{1+x^{4i}}{1-x^{2i}}=
\frac{1}{2}\sum_{k=0}^\infty\frac{(-\mathbf{i};x^4)_k(\mathbf{i};x^4)_k}{(x^2;x^4)_k(x^4;x^4)_k}x^{2k}
\\
&=&\frac{1}{2} \frac{(-\mathbf{i}x^2;x^4)_\infty (\mathbf{i}x^2;x^4)_\infty}
{(x^2;x^4)_\infty (x^2;x^4)_\infty}= \frac{1}{2}(-x^4;x^8)_\infty (-x^2;x^2)^2_\infty.
\end{eqnarray*}

\end{proof}

By~\eqref{eqn-hecke1} and~\eqref{eqn-heckeD} we have
\ber
\label{fn1B}&&\sum_{m}f_{m,1}^1x^{m}=\prod_{s\geq 1}(1+x^{2s})^2(1+x^{s})^2,\,\sum_mf_{m,0}^1x^m=\prod_{s\geq 1}(1+x^{2s-1})^2(1+x^{s})^2\nonumber\\
\label{fn2B}&&\sum_{m}f_{m,1}^2x^{m}=\frac{1}{2}\prod_{s\geq 1}(1+x^{2s})^2(1+x^{s})^2+\frac{3}{2}\prod_{s\geq 1}(1+x^{4s})(1+x^{2s})
\\
\label{fn1D}&&\sum_mf_{m,0}x^m=\frac{1}{4}\prod_{s\geq 1}(1+x^{2s-1})^2(1+x^{s})^2+\frac{3}{2}\prod_{s\geq 1}(1+x^{4s-2})(1+x^{2s})\\
\label{fn3D}&&\sum_mf_{m,0}^2x^m=\frac{1}{2}\prod_{s\geq 1}(1+x^{2s-1})^2(1+x^{s})^2+\frac{3}{2}\prod_{s\geq 1}(1+x^{4s-2})(1+x^{2s}).\nonumber
\eer
Let us write $\displaystyle{F(x)=\prod_{s\geq 1}\frac{1}{1-x^{2s}}=\sum\mathbf{p}(k)x^{2k}}$. It follows from~\eqref{T1},~\eqref{tbpq},~\eqref{bb-odd},~\eqref{eqn-bio} and the above equations that
\bern
&&\sum_{n\geq 0}T'_{2n+1}x^{n}=\left((\sum_{m}f_{m,1}^1x^{m})(\frac{1}{2}\sum_n\tilde{b}_{2n+1}x^n)+\frac{3}{2}\prod_{s\geq 1}(1+x^{4s})(1+x^{2s})(\sum_n{b}_{2n+1}^2x^n)\right)F(x)\\
&&\quad=\prod_{s\geq 1}\frac{(1+x^{2s})^4(1+x^{s})^4}{1-x^{2s}}+{3}\prod_{s\geq 1}\frac{(1+x^{4s})^2(1+x^{2s})(1+x^{s})^2}{1-x^{2s}}
\\
&&\sum_{n\geq 0}T'_{2n}x^{2n}=\left((\sum_{m}f_{m,0}^1x^{m})(\frac{1}{2}\sum_n\tilde{b}_{2n}x^n)+\frac{3}{2}\prod_{s\geq 1}(1+x^{4s-2})(1+x^{2s})(\sum_n{b}_{2n}^2x^n)+\frac{9}{4}\right)F(x)\\&&
\quad=\frac{1}{4}\prod_{s\geq 1}\frac{(1+x^{2s-1})^4(1+x^{s})^4}{1-x^{2s}}+\frac{3}{2}\prod_{s\geq 1}\frac{(1+x^{4s-2})^2(1+x^{2s})(1+x^{s})^2}{1-x^{2s}}+\frac{9}{4}\prod_{s\geq 1}\frac{1}{1-x^{2s}}\nonumber.\eern
 We conclude that $T'_{N}=T_{N}^0$ in view of Corollary~\ref{numbert}. This completes the proof of  the theorem.

\subsection{Proofs of Corollary~\ref{coro-cuspidal} and Corollary~\ref{nil coro-2}}

\begin{proof}[Proof of Corollary~\ref{coro-cuspidal}]In view of~\cite[Corollary 6.7]{VX2} and its proof, it suffices to consider the sheaves in $\Char_K(\Lg_1)_{\kappa_1}$. Suppose that $(G,K)$ is of type DIII and $n\geq 4$. We show that all character sheaves can be obtained from parabolic induction.  Suppose that $n=2n_0$. Let $\{e_i,i=1,\ldots,n\}$, $\{f_i,i=1,\ldots,n\}$ be a basis of $V^+$, $V^-$ respectively, such that $(e_i,f_j)=\delta_{i+j,n+1}$. Consider the $\theta$-stable parabolic subgroup $P$ such that $\pi(P)$ stabilises the flag $0\subset V_n:=\on{span}\{e_i,f_i,i\in[1,n_0]\}\subset V$. Let $L\subset P$ be the $\theta$-stable  Levi subgroup such that $\pi(L)\cong GL_{V_n}\cong GL_{2n_0}$. We have $\pi(L^\theta)\cong GL_{V_n\cap V^+}\times GL_{V_n\cap V^-}\cong GL_{n_0}\times GL_{n_0}$. Moreover,
$
\pi_1^{L^\theta}(\Ll_1)\cong B_{W_{n_0}}\times\bZ/2\bZ.
$
For each $\sigma\in\cP_2(n_0)$, consider the IC sheaf $\on{IC}(\Ll_1^{rs},\cT_{\sigma,\chi_1})$ where $\cT_{\sigma,\chi_1}$ corresponds to the $\pi_1^{L^\theta}(\Ll_1)$-representation $L_\sigma\otimes\chi_1$  where $B_{W_{n_0}} $ acts via the irreducible representation $L_\sigma$ of $W_{n_0}$ and $\bZ/2\bZ$ acts via the nontrivial character $\chi_1$. This is a character sheave in $\Char_{L^\theta}(\Ll_1)$ by~\cite[Theorem 5.1]{VX2}. As before one checks that $\on{IC}(\Lg_1^{rs},\cL_\sigma\otimes\bC_{\chi_1})\ \dsum\ \on{Ind}_{\Ll_1\subset\Lp_1}^{\Lg_1}\on{IC}(\Ll_1^{rs},\cT_{\sigma,\chi_1})$. Suppose that $(G, K)$ is of BDI. By the proof of Theorem~\ref{thm-type BD}, $\IC(\cO_{\mu_t},\cE_\phi)$, $\phi\in\widehat{A_K(\cO_\mu)}$ is a cuspidal character sheaf and  $\Char_K^{\on{cusp}}(\Lg_1)_{\kappa_1}$ is a subset of the set in Theorem~\ref{theorem-cusp} (ii). Thus the only possible $\theta$-stable Levi subgroups $L$ contained in  proper $\theta$-stable parabolic subgroups $P$ with $\on{\Char}_{L^\theta}^{cusp}(\Ll_1)_{\kappa_1}\neq\emptyset$ are such that $\pi(L)\cong GL_2^k\times Spin(N-4k)$, $k>0$, $N-4k\geq t^2$. But for such $L$ we have $K.\Lp_1\subsetneq\overline{\widecheck\cO_{m,t}}$. Theorem~\ref{theorem-cusp} then follows.

It remains to check the number of cuspidal character sheaves. The claims on $|\Char_K^{\on{cusp}}(\Lg_1)_{\kappa_0}|$ follow from~\eqref{fn2B},~\eqref{fn1D} and~\eqref{eqn-oeterms}. The claim on $|\Char_K^{\on{cusp}}(\Lg_1)_{\kappa_1}|$ follows from the definitions of $\Theta_{m,t}^{\kappa_1}$, $\eta_{m,t}$ and the fact that $\sum_n|\on{Irr}\cH_{S_n,-1}|x^n=\prod_{s\geq 1}(1+x^s)$.
\end{proof}

\begin{proof}[Proof of Corollary~\ref{nil coro-2}] The fact that the sheaves in the corollary are precisely the nilpotent support character sheaves follows from Theorem~\ref{thm-type BD}. Recall that  (see~\cite{VX})
\begin{subequations}
 \beq\label{tb1}
 \sum_{q=0}^\infty\tilde{b}_{q+2k+1,q}x^q=
\frac{1}{1+x^{2k+1}}\prod_{s\geq 1}\frac{(1+x^{2s-1})^2}{(1-x^{2s})^2},
\eeq\beq 
\sum_{q=0}^\infty\tilde{b}_{2q+2k,2q}x^{2q}=\frac{1}{1+x^{2k}}\prod_{s\geq 1}\frac{(1+x^{2s})^2}{(1-x^{2s})^2}.
\eeq
\end{subequations}
Applying  Proposition~\ref{number1} and entirely similar argument as in~\S\ref{ssec-number} (using~\eqref{T1} and~\eqref{tb1}), we obtain that 
\beq
{b}_{q+t,q}^2=\text{ Coefficient of $x^q$ in }\left\{\begin{array}{ll}\displaystyle{
\frac{(1+x^{t})}{(1+x^{2t})}\prod_{s\geq 1}\frac{1+x^{4s-2}}{(1-x^{2s})^2}}&\text{ if $t$ is odd}
\\
\displaystyle{\frac{(1+x^{t})}{(1+x^{2t})}\prod_{s\geq 1}\frac{1+x^{4s}}{(1-x^{2s})^2}}&\text{ if $t$, $q$ both even}.\end{array}\right.
\eeq
The claims on $|\Char_K^{\rn}(\Lg_1)_{\kappa_0}|$ then follow from~\eqref{set of chars}, ~\eqref{tbpq} and~\eqref{tb1}. 

The claim on $|\Char_K^{\rn}(\Lg_1)_{\kappa_1}|$ follows from Lemma~\ref{lemma-component gp} and the definition of $\eta_{0,t}$.
\end{proof}

\end{document}